\documentclass[12pt,reqno]{amsart}

\usepackage{amsmath,amsthm,amssymb,hyperref}
\usepackage{a4wide}
\usepackage{calrsfs}

\numberwithin{equation}{section}

\newtheorem{theorem}{Theorem}[section]
\newtheorem{lemma}[theorem]{Lemma}
\newtheorem{proposition}[theorem]{Proposition}
\newtheorem{corollary}[theorem]{Corollary}

\theoremstyle{remark}
\newtheorem{remark}[theorem]{Remark}

\theoremstyle{remark}
\newtheorem{example}[theorem]{Example}

\DeclareMathOperator{\ad}{ad}
\DeclareMathOperator{\Ad}{Ad}
\DeclareMathOperator{\Aut}{Aut}

\DeclareMathOperator{\id}{id}

\DeclareMathOperator{\mysp}{\g{sp}}

\DeclareMathOperator{\Pu}{Pu}

\DeclareMathOperator{\so}{\g{so}}
\DeclareMathOperator{\SO}{SO}
\DeclareMathOperator{\Sp}{Sp}
\DeclareMathOperator{\spann}{span}
\DeclareMathOperator{\spin}{\g{spin}}
\DeclareMathOperator{\Spin}{Spin}
\DeclareMathOperator{\su}{\g{su}}
\DeclareMathOperator{\SU}{SU}
\DeclareMathOperator{\tr}{tr}

\DeclareMathOperator{\U}{U}

\newcommand{\C}{\ensuremath{\mathbb{C}}}
\newcommand{\Ca}{\ensuremath{\mathbb{O}}}
\newcommand{\eS}{\ensuremath{\mathrm{S}}}
\newcommand{\g}[1]{\operatorname{\ensuremath{\mathfrak{#1}}}}
\newcommand{\G}{\ensuremath{\mathrm{F}_4^*}}
\newcommand{\GL}{\ensuremath{\mathrm{GL}}}

\newcommand{\HH}{\ensuremath{\mathrm{H}}}

\newcommand{\LG}{\ensuremath{\mathrm{G}}}
\newcommand{\M}{M}
\newcommand{\maxpar}{\g{k}_0 + \g{a} + \g{n}}

\newcommand{\PP}{\ensuremath{\mathrm{P}}}
\newcommand{\R}{\ensuremath{\mathbb{R}}}

\renewcommand\Re{\mathrm{Re}}
\renewcommand{\H}{\ensuremath{\mathbb{H}}}
\renewcommand{\O}{\mathrm{O}}

\begin{document}


\title[Octonions, triality, F4, and polar actions]{Octonions, triality, the exceptional Lie algebra F4, and polar actions on the Cayley hyperbolic plane}

\author[A.\ Kollross]{Andreas Kollross}
\address{Institut f\"{u}r Geometrie und Topologie, Universit\"{a}t Stuttgart, Germany}
\email{kollross@mathematik.uni-stuttgart.de}

\date{\today}

\begin{abstract}
Using octonions and the triality property of~Spin(8),
we find explicit formulae for the Lie brackets of the exceptional simple real Lie algebras $\mathfrak{f}_4$ and $\mathfrak{f}^*_4$, i.e.\ the Lie algebras of the isometry groups of the Cayley projective  plane and the Cayley hyperbolic plane. As an application, we classify polar actions on the Cayley hyperbolic plane which leave a totally geodesic subspace invariant.
\end{abstract}



\maketitle


\section{Introduction}
\label{sec:intro}


Exceptional Lie groups, i.e.\ the simple Lie groups of types~$\mathrm{E}_6$, $\mathrm{E}_7$, $\mathrm{E}_8$, $\mathrm{F}_4$, $\mathrm{G}_2$, first discovered by Killing~\cite{killing} at the end of the 19th century, appear in many contexts in geometry and physics. However, they still remain somewhat elusive objects of study, except for groups of type~$\mathrm{G}_2$, which are the exceptional simple Lie groups of lowest dimension. For the other simple exceptional Lie groups, the group structure is not directly accessible for computations and one has often to resort to some rather indirect methods in order to study geometric problems involving one of these groups, cf.~\cite{k09}.

In this article, we study isometric Lie group actions on the Cayley hyperbolic plane \[M = \Ca \HH^2 = \mathrm{F}_4^*/\Spin(9),\] a non-compact Riemannian symmetric space of rank one and dimension~$16$. More precisely, our goal is to classify \emph{polar actions} on~$\M$.
A proper isometric action of a Lie group on a Riemannian manifold is called \emph{polar} if there exists a \emph{section} for the action, i.e.\ an embedded submanifold which meets all orbits of the group action and meets them orthogonally. See \cite{B11,D12,Th05,Th10} for survey articles in connection with polar actions.
Since the criterion for polarity of an isometric action (see Proposition~\ref{prop:criterion} below) involves the Lie bracket of the isometry group, it is necessary to have a good model for the Lie algebra~$\g{f}_4^*$ in order to decide whether a given isometric Lie group action on~$\M$ is polar or not.

A widely used construction of the compact Lie group~$\mathrm{F}_4$ is to define it as the group of automorphisms of the Albert algebra, i.e.\ the exceptional Jordan algebra given by the set of self-adjoint $3 \times 3$-matrices with entries from the Cayley numbers~$\Ca$, where the multiplication is defined by $x \circ y = \frac12(xy+yx)$, see~\cite{freudenthal,murakami,ChS,adams,Baez02,figueroa,yokota}.  Another approach to construct the Lie algebra~$\g{f}_4$ is to view it as the direct sum of~$\spin(9)$ and the $16$-dimensional module of the spin representation of~$\Spin(9)$, see~\cite{Baez02,ms}. In \cite{adams}, $\mathrm{F}_4$ is defined as a subgroup of~$\mathrm{E}_8$. While these various constructions provide existence proofs for the Lie algebras of type~$\mathrm{F}_4$, it appears that they are not so well suited to study our particular geometric problem involving Lie group actions on~$\Ca \HH^2$, where it is necessary to compute Lie brackets explicitly and to have good descriptions of Lie triple systems and subalgebras.
In~\cite[Section~4.2]{Baez02}, Baez remarks that the isomorphism
\[
\g{f}_4 \cong \so(\Ca) \oplus \Ca^3
\]
is an elegant way of describing~$\g{f}_4$. He writes:\ \textit{``This formula emphasizes the close relation between~$\g{f}_4$ and triality: the Lie bracket in~$\g{f}_4$ is completely built out of maps involving~$\so(8)$ and its three $8$-dimensional irreducible representations!''} In this article, we pursue this approach and obtain the explicit expressions~(\ref{eq:cbracket}) and~(\ref{eq:ncbracket}) for the bracket of the compact real form~$\g{f}_4$ and the non-compact real form~$\g{f}_4^*$. These formulae involve octonionic multiplication and the triality automorphism of the Lie algebra~$\so(8)$. We use them to describe various subalgebras of~$\g{f}_4^*$, as well as the totally geodesic subspaces of~$\Ca \HH^2$, and to decide whether or not certain subgroups of~$\mathrm{F}_4^*$ act polarly on~$\Ca \HH^2$. Note that  our model also encompasses the description $\g{f}_4 \cong \spin(9) \oplus \R^{16}$, since we have $\so(8) \oplus \Ca \cong \so(9)$, cf.\ Lemma~\ref{lm:sonine} below.

Our constructions provide a new existence proof for the exceptional simple Lie algebras~$\g{f}_4$ and~$\g{f}_4^*$.

As an application, the following result on polar actions is proved. (For a more detailed description of the actions involved here, see Table~\ref{tbl:actions} or Corollary~\ref{cor:main}.)

\begin{theorem}\label{th:main}
Let $H$ be a closed connected non-trivial subgroup of~$\G$ whose action on~$\Ca\HH^2$ is polar and leaves a totally geodesic subspace~$P \neq \Ca\HH^2$ invariant. Then

\begin{enumerate}

\item either the action has a fixed point and $H$ is conjugate to one of
\begin{equation}\label{eq:fixpt}
\Spin(9), \quad \Spin(8), \quad \Spin(7) \cdot \SO(2), \quad \Spin(6) \cdot \Spin(3)
\end{equation}

\item or the $H$-orbits coincide with the orbits of the connected component~$N(P)_0$ of the normalizer $N(P) = \left \{ g \in \G \,\colon g \cdot P = P \right \}$ where $P$ is not isometric to a real hyperbolic space of dimension~$3$ or~$4$.

\end{enumerate}
Conversely, the actions of the groups~\emph{(\ref{eq:fixpt})} on~$\Ca\HH^2$ are polar and, furthermore, if $P$ is a totally geodesic subspace of~$\M$ not isometric to a real hyperbolic space of dimension~$3$ or~$4$, then $N(P)_0$ acts polarly on~$\M$.

\end{theorem}

It should be noted that the ``Conversely,~\dots''-part of the statement was already known, in fact, it follows from the classification of polar actions by reductive algebraic subgroups of~$\G$ which was obtained in~\cite{k11}. The result in~\cite{k11} was proved using the duality of Riemannian symmetric spaces and the classification of polar actions on~$\Ca \PP^2$ by~Podest\`{a} and Thorbergsson~\cite{PT99}, see also~\cite{gk16}. Polar actions with a fixed point on irreducible symmetric spaces have been classified in~\cite{DK10}.

Moreover, actions of cohomogeneity one on various Riemannian symmetric spaces of non-compact type have been classified in \cite{B98,BB01,BT03,BT04,BT07}. In particular, in the article~\cite{BT07} by~Berndt and~Tamaru, a complete classification of cohomogeneity one actions on~$\Ca \HH^2$ was obtained.

Let us also remark that polar actions on Riemannian symmetric spaces of compact type have been completely classified \cite{k02,k07,k09,L11,KL12}, see~\cite[Introduction]{KL12} for a summary. However, for symmetric spaces of non-compact type, polar actions so far have only been completely classified in the case of real hyperbolic spaces~\cite{wu} and complex hyperbolic spaces~\cite{DDK12}.

The new results in our main theorem can also be stated more explicitly in the following form.

\begin{corollary}\label{cor:main}
Let $H$ be a closed connected non-trivial subgroup of~$\G$ whose action on~$\Ca\HH^2$ is polar and leaves a totally geodesic subspace~$P \neq \Ca\HH^2$ of positive dimension invariant. Then there exists a maximal connected subgroup~$L$ of~$\G$ containing $H$ such that the $L$-orbits and the $H$-orbits coincide and $L$ is conjugate to one of the following: $\LG_2 \cdot \SO_0(1,2)$, $\SU(3) \cdot \SU(1,2)$, $\Sp(1) \cdot \Sp(1,2)$, $\Spin(7) \cdot \SO_0(1,1)$, $\Spin(6)\cdot \Spin(1,2)$, $\Spin(3)  \cdot \Spin(1,5)$, $\SO(2) \cdot \Spin(1,6)$, $\Spin(1,7)$, $\Spin(1,8)$, cf.~Table~\ref{tbl:actions}.
\end{corollary}

This article is organized as follows.
We start in Section~\ref{sec:octonions} by recalling notation and basic material on octonions and the Lie algebra~$\so(8)$ mostly from~Freudenthal~\cite{freudenthal} and Murakami~\cite{murakami}.
The next section contains the construction of the compact Lie algebra~$\g{f}_4$. We first define a skew-symmetric bracket operation~(\ref{eq:cbracket}) on~$\so(8) \times \Ca^3$ and prove that the  algebra given in this way indeed satisfies the Jacobi identity. We then show that this $52$-dimensional real Lie algebra is simple and has an $\ad$-invariant scalar product; thus it is isomorphic to the Lie algebra of the compact Lie group~$\mathrm{F}_4$.
The formula for the bracket of the Lie algebra of the isometry group $\mathrm{F}_4^*$ of~$\M$ can now be obtained in Section~\ref{sec:isom} simply by using the duality of Riemannian symmetric spaces. We describe the restricted root space decomposition of~$\M$, an Iwasawa decomposition of~$\g{g}^*$, and a maximal parabolic subalgebra of~$\g{g}^*$ in the framework of our model for~$\g{f}_4^*$.
In Section~\ref{sec:totgeod} we study totally geodesic subspaces of~$\M$ and write down Lie triple systems corresponding to each congruence class.
We recall a criterion for polarity of an isometric Lie group action on a non-compact Riemannian symmetric space in Section~\ref{sec:polact} and give some examples of polar actions, in particular, we find the new Example~\ref{ex:coh2nilp}.
In Section~\ref{sec:totgorb} we classify polar actions on~$\M$ leaving a totally geodesic subspace invariant.
We mention some open questions and possible generalizations of this work in the last section.


\section{\texorpdfstring{Octonions and~$\so(8)$}{Octonions and so(8)}}\label{sec:octonions}


Recall that the \emph{octonions}~$\Ca$, also called the \emph{Cayley numbers}, are an $8$-dimensional real division algebra, which is neither commutative nor associative. Let $e_0, \dots, e_7 \in \Ca$ be basis vectors such that $e := e_0 = 1$ is the multiplicative identity, $e_1^2 = \dots = e_7^2 = -1$, and
\begin{equation}\label{eq:omult}
\begin{array}{llll}
  e_1e_2=e_3,\quad & e_1e_4=e_5,\quad & e_2e_4=e_6,\quad & e_3e_4=e_7, \\
  e_5e_3=e_6, & e_6e_1=e_7, & e_7e_2=e_5. &
\end{array}
\end{equation}
Using this basis, we identify the Cayley numbers $\Ca$ with $\R^8$.
For $a \in \Ca$, let $L_a$ and $R_a$ be the maps $\R^8 \to \R^8$ given as left and right multiplication by~$a$, i.e.\ $L_a(x) = ax$, $R_a(x) = xa$.
Recall that the octonions are an \emph{alternative} algebra, i.e.\ the alternative laws $(xx)y=x(xy)$ and $(xy)y=x(yy)$ hold for all $x,y \in \Ca$; or equivalently, the \emph{associator}, i.e.\ the $\R$-trilinear map $[\cdot,\cdot,\cdot] \colon \Ca \times \Ca \times \Ca \to \Ca$, defined by $[a,b,c] = (ab)c-a(bc)$, is an alternating map.
Using the alternative laws, the whole octonionic multiplication table may be recovered from the above identities. Let octonionic \emph{conjugation} $\gamma \colon \Ca \to \Ca,\; x \mapsto \bar x$ be the $\R$-linear map defined by
\[
\bar e_i =
\left\{
  \begin{array}{ll}
    e_0, & \hbox{if~$i = 0$;} \\
    -e_i, & \hbox{if~$i \ge 1$.}
  \end{array}
\right..
\]
The map $\gamma$ is an involutive antiautomorphism of the Cayley numbers, i.e.\ we have $\gamma^2 (x)= x$ and $\overline{xy} = \bar y \bar x$ for all $x,y \in \Ca$.
Let the \emph{pure part} of an octonion be defined by $\Pu(x) = \frac12(x-\bar x)$. Let the \emph{real part} of an octonion~$x$ be the real number~$\Re(x)$ defined by $x = \Re(x)e+\Pu(x)$.

Let $\so(8)$ be the Lie algebra of skew symmetric real $8 \times 8$-matrices with the commutator of matrices~$[A,B] = AB-BA$ as the bracket. We identify $\R^8 \wedge \R^8$  with  $\so(8)$ using the definition
\[
x \wedge y = x y^t - y x^t,
\]
where, on the right hand side, we consider elements of~$\Ca$ as column vectors in~$\R^8$, $x^t$~denotes the row vector which is the transpose of~$x$, and the usual matrix multiplication is understood.

Following Freudenthal~\cite{freudenthal} and Murakami~\cite[\S2]{murakami}, we define
$\pi \in \Aut(\so(8))$ by
\[
\pi(a \wedge b) = \tfrac12 L_b \circ L_a \hbox{~for~$a \in \Pu(\Ca), b \in \Ca$}.
\]
It can be checked by a direct (if somewhat cumbersome) calculation that~$\pi$ is actually an automorphism of~$\so(8)$.
Furthermore, we define another automorphism~$\kappa$ of~$\so(8)$ by $\kappa(a \wedge b) = \bar a \wedge \bar b$, or equivalently, by
\[
\kappa(A)(x) = \overline{A(\bar x)}
\]
Finally, let
\[
\lambda := \pi \circ \kappa \in \Aut(\so(8)),
\]
then we have $\pi^2 = \kappa^2 = \lambda^3 = 1$ and $\kappa \circ \lambda^2 = \lambda \circ \kappa = \pi$, see~\cite[\S2, Thm.~2]{murakami}.
More explicitly, the automorphism~$\lambda$ is given by
\[
\lambda(a \wedge b) = \tfrac12 L_{\bar b} \circ L_{\bar a} \hbox{~for~$a \in \Pu(\Ca), b \in \Ca$}
\]
and it follows that its square $\lambda^2 = \kappa \circ \pi$ is given by
\[
\lambda^2(a \wedge b) = \tfrac12 R_{\bar b} \circ R_{\bar a} \hbox{~for~$a \in \Pu(\Ca), b \in \Ca$}
\]
since $\lambda^2(a \wedge b)(x) = \kappa(\pi(a \wedge b))(x) = \frac12 \kappa (L_b \circ L_a)(x) =  \frac12 \overline{b(a \bar x)} = \frac12 (x \bar a)\bar b.$

Let $T_a = R_a + L_a$ for $a \in \Ca$, then (cf.~\cite[p.188]{murakami}) we have for $i \ge 1$ that
\[
T_{e_i} (x) = x e_i + e_i x =
\left\{
  \begin{array}{ll}
    -2e_0, & \hbox{if~$x=e_i$;} \\
    2e_i, & \hbox{if~$x=e_0$;} \\
    0, & \hbox{if~$x=e_j$, $0 \neq j \neq i$;}
  \end{array}
\right.
\]
and it follows that
\begin{equation}\label{eq:Gij}
T_{e_i}=2e_ie_0^t-2e_0e_i^t = 2e_i \wedge e_0.
\end{equation}

\begin{lemma}\label{lm:sonine}
Let $\so(8) \times \R^8$ be equipped with the binary operation given by
\[
[(A,x),(B,y)] = (AB-BA-4x \wedge y, Ay-Bx).
\]
Then $\so(8) \times \R^8$ is a real Lie algebra isomorphic to~$\so(9)$.
\end{lemma}

\begin{proof}
We check that an isomorphism $\so(8) \times \R^8 \to \so(9)$ is given by the map
\[
(A,x) \mapsto
\begin{pmatrix}
  A & 2 x \\
  -2 x^t & 0 \\
\end{pmatrix}.
\]
Computing the bracket in~$\so(9)$ we get
\begin{align*}
\left[\begin{pmatrix}
  A & 2 x \\
  -2 x^t & 0 \\
\end{pmatrix},
\begin{pmatrix}
  B & 2 y \\
  -2 y^t & 0 \\
\end{pmatrix}\right]=
\begin{pmatrix}
  AB-BA-4xy^t+4yx^t & 2Ay-2Bx \\
  -2 x^tB +2y^tA& 0 \\
\end{pmatrix}.
\end{align*}
This shows that the map defined above is indeed an automorphism.
\end{proof}


\section{\texorpdfstring{The compact Lie algebra~$\g{f}_4$}{The compact Lie algebra f4}}\label{sec:f4cpt}


In this section, we will describe an explicit construction of $\g{f}_4$ which is based on the inclusion $\Spin(8) \subset \rm F_4$. It is well known, see e.g.~\cite[Tables~25, 26]{dynkin1}, that the isotropy representation of the homogeneous space $\rm F_4 / \Spin(8)$ is the direct sum of the three mutually inequivalent $8$-dimensional irreducible representations of~$\Spin(8)$. These representations can be conveniently described using octonions, see~\cite{murakami}, and we will define a Lie algebra structure on ${\mathcal A} = \so(8) \times \Ca^3$ such that the action of $\so(8)$ on~$\Ca \times \Ca \times \Ca$ is equivalent to $\rho \oplus \rho \circ \lambda \oplus \rho \circ \lambda^2$, where $\rho$ is the standard representation of $\so(8)$, and such that
\[
\tau \colon (A,x,y,z) \mapsto (\lambda(A),y,z,x)
\]
is an automorphism of order~3 of ${\mathcal A}$.

Let us define a skew-symmetric bracket \[[\cdot,\cdot]\colon {\mathcal A} \times {\mathcal A} \to {\mathcal A}\] as follows.
The bracket on $\so(8)$ is defined in the obvious way, i.e.\
\[
[(A,0,0,0),(B,0,0,0)] = (AB-BA,0,0,0).
\]
We identify $\Ca$ with $\R^8$ using the basis $e_0, \dots, e_7$ as above and let
\[
[(A,0,0,0),(0,x,0,0)] = (0,Ax,0,0).
\]
Using the identification $\so(8) = \R^8 \wedge \R^8$ as above, we may define the bracket on $\{0\} \times \Ca \times \{0\} \times \{0\}$ by
\[
[(0,x,0,0),(0,y,0,0)] = (-4x \wedge y,0,0,0).
\]
In this way, $\so(8) \times \Ca \times \{0\} \times \{0\}$ becomes a subalgebra isomorphic to~$\so(9)$, see Lemma~\ref{lm:sonine}.
We further define, using octonionic multiplication and conjugation,
\begin{equation*}\label{eq:xybrak}
[(0,x,0,0),(0,0,y,0)] = (0,0,0,\overline{xy}).
\end{equation*}
Now all brackets on ${\mathcal A}$ are defined by requiring that the map $\tau$ be an algebra automorphism with respect to the bracket operation and extending to a bilinear and skew symmetric map ${\mathcal A} \times {\mathcal A} \to {\mathcal A}$. An explicit formula for the bracket operation is given in the statement of the following theorem.

\begin{theorem}\label{thm:f4jacobi}
The bracket operation on ${\mathcal A}$as defined above, i.e.\ the map
\begin{align}\label{eq:cbracket}
[(A,u,v,w),(B,x,y,z)] = (C,r,s,t)
\end{align}
where
\begin{align*}
C &= AB-BA-4u \wedge x - 4\lambda^2(v \wedge y) - 4\lambda(w \wedge z),\\
r &= Ax-Bu+\overline{vz}-\overline{yw},\\
s &= \lambda(A)y-\lambda(B)v+\overline{wx}-\overline{zu},\\
t &= \lambda^2(A)z-\lambda^2(B)w+\overline{uy}-\overline{xv},
\end{align*}
is $\R$-bilinear, skew symmetric and satisfies the Jacobi identity. The real Lie algebra $\left({\mathcal A},[{\cdot},{\cdot}]\right)$ defined in this fashion is isomorphic to the Lie algebra of the compact exceptional simple Lie group of type~$\rm F_4$.
\end{theorem}

The $\R$-bilinearity and skew symmetry of the bracket are obvious from the definition. To prove the rest of Theorem~\ref{thm:f4jacobi}, we will use Lemmas~\ref{lm:tautheta}, \ref{lm:spin8aut}, \ref{lm:f4jacobi} and~\ref{lm:killing} below.

Let $\rho$ be the standard representation $\rho \colon \Spin(8) \to \SO(8)$. Define an action of~$\Spin(8)$ on $\R^8$ by $\theta(x) := \rho(\theta)(x)$. Since $\Spin(8)$ is simply connected, its automorphism group is canonically isomorphic to the automorphism group of~$\so(8)$. The automorphism of~$\Spin(8)$ thus given by $\lambda \in \Aut(\so(8))$ is also denoted by~$\lambda$. Define an action of~$\Spin(8)$ on $\mathcal A$ by
\[
 \theta  (A,x,y,z) = (\Ad_\theta(A), \theta(x), \lambda(\theta)(y), \lambda^2(\theta)(z) )
\]
for $\theta \in \Spin(8)$.

\begin{lemma}\label{lm:tautheta}
We have
\[
\tau( \theta(A,x,y,z) ) = \lambda(\theta)(\tau(A,x,y,z))
\]
for all $\theta \in \Spin(8)$, $(A,x,y,z) \in \mathcal A$.
\end{lemma}

\begin{proof}
We compute
\begin{align*}
\tau( \theta(A,x,y,z) ) &= \tau(\Ad_\theta(A), \theta(x), \lambda(\theta)(y), \lambda^2(\theta)(z) )=\\
&=(\Ad_{\lambda(\theta)}(\lambda(A)), \lambda(\theta)(y), \lambda^2(\theta)(z),  \theta (x)) =\\
&=\lambda(\theta)(\lambda(A), y, z, x) =\\
&=\lambda(\theta)(\tau(A, x, y, z)),
\end{align*}
where we have used that $\lambda(\Ad_\theta(A))=\Ad_{\lambda(\theta)}(\lambda(A))$ for all $\theta \in \Spin(8)$, $A \in \so(8)$.
\end{proof}

\begin{lemma}\label{lm:spin8aut}
The map \[(A,x,y,z) \mapsto \theta  (A,x,y,z)\] is an automorphism of the real algebra~$\mathcal A$ for any $\theta \in \Spin(8)$.
\end{lemma}

\begin{proof}
Obviously, we have
\begin{align*}
&[\theta(A,0,0,0),\theta(B,0,0,0)] =
[(\Ad_\theta(A),0,0,0),(\Ad_\theta(B),0,0,0)] = \\
&=(\Ad_\theta([A,B]),0,0,0) =
\theta[(A,0,0,0),(B,0,0,0)].
\end{align*}
Using that $\lambda(\Ad_\theta(A))=\Ad_{\lambda(\theta)}(\lambda(A))$ for all $\theta \in \Spin(8)$, $A \in \so(8)$, we get
\begin{align*}
&[\theta(A,0,0,0),\theta(0,x,y,z)]
=[(\Ad_\theta(A),0,0,0),(0, \theta  (x), \lambda(\theta)  (y), \lambda^2(\theta) (z))] = \\
&= (0,\Ad_\theta(A)(\theta  (x)), \lambda(\Ad_\theta(A))(\lambda(\theta)(y)), \lambda^2(\Ad_\theta(A))(\lambda^2(\theta)(z))) =\\
&=(0,\theta(Ax),\lambda(\theta)(\lambda(A)y), \lambda^2(\theta)(\lambda^2(A)z)) =\\
&=\theta(0,Ax,\lambda(A)y, \lambda^2(A)z)=\theta[(A,0,0,0),(0,x,y,z)].
\end{align*}
By the \emph{principle of triality}, see~\cite[\S2, \S3]{murakami}, we have that
\[
\theta(x)\lambda(\theta)(y) = \kappa\circ\lambda^2(\theta)(xy)
\]
for all $x,y \in \Ca$. From this, we obtain
\begin{align*}
&[\theta(0,u,0,0),\theta(0, 0, y, 0 )]
=[(0, \theta(u),0,0),(0, 0, \lambda(\theta)(y), 0 )]=\\
&= (0,0,0,\overline{\theta(u)\lambda(\theta)(y)})
= (0,0,0,\overline{\kappa\circ\lambda^2(\theta)(uy)})=\\
&= (0,0,0,\lambda^2(\theta)(\overline{uy}))=\theta(0,0,0,\overline{uy})=\theta[(0,u,0,0),(0, 0,y, 0 )],
\end{align*}
where we have used that $\kappa(\theta)(x) = \overline{\theta(\bar x)}$ for all $\theta \in \Spin(8), x \in \Ca$.

Furthermore, we compute
\begin{align*}
&[\theta(0,u,0,0),\theta(0,x,0,0 )]
=[(0, \theta(u),0,0),(0, \theta(x), 0, 0 )]=\\
&= (-4\theta(u)\wedge\theta(x),0,0,0)=\\
&= (-4\theta (u)\theta(x)^t + 4\theta (x) \theta(u)^t,0,0,0)
=\\
&= (-4\Ad_{\theta}(u \wedge x),0,0,0) =
\theta [(0,u,0,0),(0,x,0,0 )].
\end{align*}
The statement of the lemma now follows using Lemma~\ref{lm:tautheta} and the skew symmetry and bilinearity of the bracket.
\end{proof}

\begin{lemma}\label{lm:f4jacobi}
The bracket operation on~$\mathcal A$ satisfies the Jacobi identity, i.e.\ we have
\begin{equation}\label{eq:Jacobi}
[\xi,[\eta,\zeta]]+[\eta,[\zeta,\xi]]+[\zeta,[\xi,\eta]]=0
\end{equation}
for all $\xi, \eta, \zeta \in \mathcal A$.
\end{lemma}

\begin{proof}
By the trilinearity of the left-hand side, it suffices to show that (\ref{eq:Jacobi}) holds for all triples of vectors $(\xi,\eta,\zeta)$ where each of the elements $\xi,\eta,\zeta$ is a vector from one of the factors of $\so(8) \times \Ca \times \Ca \times \Ca$. Let us assume the vectors $\xi,\eta,\zeta$ are chosen in this fashion.
We have not yet shown that~$\mathcal A$ is a Lie algebra with the bracket as defined above, but in any case, the map~$\tau$ certainly is an automorphism with respect to whatever real algebra structure is defined on~$\mathcal A$ by the bracket operation~(\ref{eq:cbracket}). Using this fact, we may assume that the three vectors $\xi,\eta,\zeta$ are either from $\so(8) \times \Ca \times \Ca \times \{0\}$ or from $\{0\} \times \Ca \times \Ca \times \Ca$.

\begin{enumerate}

\item First assume $\xi,\eta,\zeta \in \so(8) \times \Ca \times \Ca \times \{0\}.$
If the three vectors $\xi,\eta,\zeta$ are from $\so(8) \times \Ca \times \{0\} \times \{0\}$ or $\so(8) \times \{0\} \times \Ca \times \{0\}$, it follows from Lemma~\ref{lm:sonine} (using the fact that $\tau$ is an automorphism of~$\mathcal A$ in the second case) that the Jacobi identity holds for $\xi,\eta,\zeta$. Thus may assume
$
\xi = (A,0,0,0), 
\eta = (0,x,0,0), 
\zeta = (0,0,y,0).
$ 
To verify the Jacobi identity in this special case, we compute:
\begin{align*}
[(A,0,0,0),[(0,x,0,0),(0,0,y,0)]]&=(0,0,0,\lambda^2(A)(\overline{xy})), \\
[(0,x,0,0),[(0,0,y,0),(A,0,0,0)]]&=(0,0,0,-\overline{x \cdot \lambda(A)(y)}),\\
[(0,0,y,0),[(A,0,0,0),(0,x,0,0)]]&=(0,0,0,-\overline{A(x) \cdot y}).
\end{align*}
The sum of the fourth components of these elements is the conjugate of
\[
\kappa \circ \lambda^2(A)(xy)-x \cdot \lambda(A)(y) -A(x) \cdot y.
\]
It follows from the \emph{infinitesimal principle of triality}, see~\cite[\S2, Thm.~1]{murakami}, that the above expression is zero for all $A \in \so(8)$ and all $x,y \in \Ca$.

\item Now assume $\xi,\eta,\zeta \in \{0\} \times \Ca \times \Ca \times \Ca.$
First consider the subcase where
$
\xi = (0,x,0,0), 
\eta = (0,0,y,0), 
\zeta = (0,0,0,z)
$
are unit vectors.
The group $\Spin(8)$ acts as a group of automorphisms on~$\mathcal A$ by Lemma~\ref{lm:spin8aut} and we may use this action to assume $x=e$ and $y=e$, since $\Spin(8)$ acts transitively on the product of unit spheres~$\eS^7 \times \eS^7 \subset \R^8 \oplus \R^8$ by the sum of any two of its inequivalent irreducible 8-dimensional representations.
Using this assumption, we compute
\begin{align*}
[(0,x,0,0),[(0,0,y,0),(0,0,0,z)]] &= (-4e \wedge \bar z,0,0,0), \\
[(0,0,y,0),[(0,0,0,z),(0,x,0,0)]] &= (-4\lambda^2(e \wedge \bar z),0,0,0),\\
[(0,0,0,z),[(0,x,0,0),(0,0,y,0)]] &= (-4\lambda(z \wedge e),0,0,0).
\end{align*}
Since all three terms are zero if $z=e$, we may assume that $z \in \Pu(\Ca)$.
The sum of the first components of these elements is then
\begin{align*}
-&4e \wedge \bar z - 4 \lambda^2(e \wedge \bar z) -4\lambda(z \wedge e) = \\
&=4e \wedge z + 4\lambda^2(\bar z \wedge e)-4\lambda(z \wedge e) = \\
&=4e \wedge z + 2 R_{e} \circ R_z - 2 L_{e} \circ L_{\bar z} = 4e \wedge z +2R_z+2L_z=0,
\end{align*}
which is zero by~(\ref{eq:Gij}).
Now consider
$
\xi = (0,x,0,0), 
\eta = (0,y,0,0), 
\zeta = (0,0,z,0).
$
We compute
\begin{align*}
[(0,x,0,0),[(0,y,0,0),(0,0,z,0)]] &= (0,0,-\bar x(yz),0), \\
[(0,y,0,0),[(0,0,z,0),(0,x,0,0)]] &= (0,0,\bar y (xz),0),\\
[(0,0,z,0),[(0,x,0,0),(0,y,0,0)]] &= (0,0,4\lambda(x \wedge y)z,0).
\end{align*}
Using the $\Spin(8)$-action again, we may assume $x=z=e$ and $y \in \Pu(\Ca)$. Then we have $-\bar x(yz)=-y$, $\bar y (xz)=-y$, and
\[
4\lambda(x \wedge y)z=4\lambda(e \wedge y)e=-4\lambda(y \wedge e)e=-2L_{e} \circ L_{\bar y}(e)=2y,
\]
showing that the sum of the above elements is zero.
Finally, assume
$
\xi = (0,x,0,0), 
\eta = (0,0,y,0), 
\zeta = (0,0,z,0).
$ 
Then we obtain
\begin{align*}
[(0,x,0,0),[(0,0,y,0),(0,0,z,0)]] &= (0,4\lambda^2(y \wedge z)x,0,0), \\
[(0,0,y,0),[(0,0,z,0),(0,x,0,0)]] &= (0,-(xz)\bar y ,0,0),\\
[(0,0,z,0),[(0,x,0,0),(0,0,y,0)]] &= (0,(xy)\bar z,0,0).
\end{align*}
Once more, using the $\Spin(8)$-action, we may assume that $x=y=e$ and $z \in \Pu(\Ca)$. Then we have \[
4\lambda^2(y \wedge z)x= 4\lambda^2(e \wedge z)e= -4\lambda^2(z \wedge e)e=-2 R_{e} \circ R_{\bar z}(e)=2z,
\]
and $-(xz)\bar y=-z$, $(xy)\bar z=-z$. This shows that the sum of the above three elements is zero.

\end{enumerate}
Using the skew-symmetry of the bracket and the fact that $\tau$ is an automorphism of~$\mathcal A$, it now follows that the Jacobi identity holds for the bracket operation on~$\mathcal A$.
\end{proof}

We define a scalar product on the Lie algebra~$\mathcal A$ by
\begin{equation}\label{eq:kill}
\langle (A,u,v,w) , (B,x,y,z) \rangle = 8(u^tx+v^ty+w^tz)-\tr(AB).
\end{equation}
This scalar product is invariant under~$\tau$. It is also $\Spin(8)$-invariant, i.e.\ we have
\[
\langle \theta(A,u,v,w) , \theta(B,x,y,z) \rangle = \langle (A,u,v,w) , (B,x,y,z) \rangle
\]
for all $(A,u,v,w),(B,x,y,z)\in\mathcal A$ and all $\theta\in\Spin(8)$.

\begin{lemma}\label{lm:killing}
The scalar product~(\ref{eq:kill}) on the Lie algebra~$\mathcal A$ is $\ad$-invariant.
\end{lemma}

\begin{proof}
Let $C \in \so(8)$. We compute
\begin{align*}
&\langle [(A,u,v,w),(C,0,0,0)],(B,x,y,z) \rangle = \\
&=\langle ([A,C],-Cu,-\lambda(C)v,-\lambda^2(C)w),(B,x,y,z) \rangle = \\
&=8(-(Cu)^t x-(\lambda(C)v)^t y-(\lambda^2(C)w)^t z)-\tr(ACB-CAB)=\\
&=8(u^tCx+v^t\lambda(C)y+w^t\lambda^2(C)z)-\tr(ACB-ABC)=\\
&=\langle (A,u,v,w),([C,B],Cx,\lambda(C)y,\lambda^2(C)z) \rangle =\\
&=\langle (A,u,v,w),[(C,0,0,0),(B,x,y,z)] \rangle.
\end{align*}
Furthermore, we have
\begin{align*}
&\langle [(A,u,v,w),( 0,e,0,0)],(B,x,y,z) \rangle = \\
&=\langle (-4u \wedge e,Ae,\bar w, - \bar v),(B,x,y,z) \rangle = \\
&=8((Ae)^tx + \bar w^t y -\bar v^t z)+4\tr((u \wedge e) B)=\\
&=8(-u(Be)^t-v^t \bar z + w^t \bar y)+4\tr(A(e \wedge x))=\\
&=\langle (A,u,v,w),(-4 e \wedge x, -Be, -\bar z, \bar y) \rangle =\\
&=\langle (A,u,v,w),[(0,e,0,0),(B,x,y,z)] \rangle,
\end{align*}
where we have used $\tr(A(e \wedge x)) = \tr(Aex^t-Axe^t) = 2\tr(Aex^t) = 2(Ae)^tx$. Since $\Spin(8)$ acts transitively on the unit sphere in $\{0\} \times \Ca \times \{0\} \times\{0\}$ and $\tau$ is an automorphism of~$\mathcal A$ which leaves the scalar product invariant, the above calculations suffice by linearity to prove $\ad$-invariance.
\end{proof}

\begin{proof}[Proof of Theorem~\ref{thm:f4jacobi}]
Observe that the real Lie algebra $\left({\mathcal A},[{\cdot},{\cdot}]\right)$ is simple. Indeed, its adjoint representation restricted to the subalgebra~$\so(8)$ acts irreducibly on each of the direct summands in $\so(8) \times \Ca \times \Ca \times \Ca$ and such that the four representation modules are mutually inequivalent. Therefore, any non-trivial ideal of~$\mathcal A$ is a sum of one or more of these modules. However, it can be easily seen from the explicit formula of the bracket operation in the statement of the theorem that the only such sum which is an ideal is $\mathcal A$ itself.

Since $\mathcal A$ is 52-dimensional, it follows from the classification of simple real Lie algebras that $\mathcal A$ is isomorphic either to the compact real form $\g{f}_{4(-52)}$ or to one of the non-compact real forms $\g{f}_{4(4)}$, $\g{f}_{4(-20)}$ of the complex simple Lie algebra of type~$\rm F_4$.
However, the scalar product on~$\mathcal A$ as defined above is positive definite and $\ad$-invariant, thus it follows that $\mathcal A$ is isomorphic to the compact form.
\end{proof}

Define $\g{k} := \so(8) \times \Ca \times \{0\} \times \{0\}$ and $\g{p} := \{0\} \times \{0\} \times \Ca \times \Ca$. We will write~$\g{g}$ instead of~$\mathcal A$ from now on. We have the decomposition $\g{g} = \g{k} \oplus \g{p}$. Since there is only one conjugacy class of a subgroup locally isomorphic to~$\Spin(9)$ in the compact Lie group~$\rm F_4$, see~\cite{dynkin1}, this decomposition corresponds to the Riemannian symmetric space $\Ca \PP^2 = {\rm F}_4 / \Spin(9)$, the compact Cayley projective plane.


\section{\texorpdfstring{The isometry group of~$\Ca \HH^2$ and its root space decomposition}{The isometry group of OH2 and its root space decomposition}}\label{sec:isom}


Using the duality of symmetric spaces~\cite[Ch.~V, \S2]{helgason}, we may define the Lie algebra~$\g{g}^*$ by $\g{g}^* = \g{k} \oplus \sqrt{-1}\,\g{p}$ as a subalgebra of the complexification $\g{g} \otimes \C$.
Since the Lie algebra structure on~$\g{g}^*$ differs from the one on~$\g{g}$ only by changing the sign of the bracket on~$\g{p} \times \g{p}$, we obtain a model for~$\g{g}^*$ which is very similar to the model for~$\g{g}$ we constructed in Section~\ref{sec:f4cpt}. Using the identifications $\g{p}^* = \{0\} \times \{0\} \times \Ca \times \Ca$ and $\g{g}^* = \so(8) \times \Ca \times \Ca \times \Ca$  we may define a new bracket on~$\g{g}^*$ by setting
\begin{align}\label{eq:ncbracket}
[(A,u,v,w),(B,x,y,z)] = (C^*,r^*,s,t)
\end{align}
where
\begin{align*}
C^* &= AB-BA-4u \wedge x + 4\lambda^2 (v \wedge y) + 4\lambda (w \wedge z),\\
r^* &= Ax-Bu-\overline{vz}+\overline{yw},\\
s &= \lambda(A)y-\lambda(B)v+\overline{wx}-\overline{zu},\\
t &= \lambda^2(A)z-\lambda^2(B)w+\overline{uy}-\overline{xv},
\end{align*}
Let $\G$ be the simply connected Lie group whose Lie algebra is~$\g{g}^*$. Let $K$ be the connected subgroup of~$\G$ whose Lie algebra is~$\g{k}$. It is well known that $\G$ is the full isometry group of~$\M$ and $K$ is the stabilizer of a point~$o$.

Let $\vartheta \colon \g{g}^* \to \g{g}^*$ be the Cartan involution corresponding to the Cartan decomposition $\g{g}^* = \g{k} + \g{p}^*$, i.e.\ the linear map defined by $\vartheta(X+Y) = X-Y$ for $X \in \g{k}$ and $Y \in \g{p}^*$. Using the identification $\g{p} = \g{p}^*$, we can use $\langle \cdot, \cdot \rangle$ also as a scalar product on~$\g{g}^*$. The Killing form of~$\g{g}^*$ is then a multiple of~$\langle \cdot, \vartheta(\cdot) \rangle$ with a negative scaling factor.

Let $\g{a}$ be the one-dimensional subalgebra of~$\g{g}^*$ spanned by~$(0,0,e,0) \in \g{p}^*$.
In order to determine the restricted root space decomposition of~$\g{g}^*$ with respect to~$\g{a}$, we compute the following bracket, assuming $y \in \Pu(\Ca)$:
\begin{align*}
[(0,0,e,0),(B,x,y,z)] &= (4\lambda^2(e \wedge y), -\bar z , -\lambda(B)e,-\bar x)=\\
&= (2R_y, -\bar z , -\lambda(B)e,-\bar x),
\end{align*}
where we have used $4\lambda^2(e \wedge y)=-4\lambda^2(y \wedge e)=-2R_{e} \circ R_{\bar y}=2R_y$. Define
\[
\g{g}_{\pm\alpha} = \{ (0,\mp \bar x,0,x)\in \g{g}^* \colon x \in \Ca \}.
\]
Furthermore, set
\begin{align*}
\g{g}_{\pm 2 \alpha} &= \{ (\pm 2\lambda^2(e \wedge p),0,p,0)\in \g{g}^* \colon p \in \Pu(\Ca) \}
=\\
&= \{ (\pm R_p,0,p,0)\in \g{g}^* \colon p \in \Pu(\Ca) \}
.
\end{align*}

Let $\so(7)$ be the subalgebra of~$\so(8)$ spanned by the elements $a \wedge b$ where $a,b \in \Pu(\Ca)$.
Define
\[
\g{k}_0 = \{ (\lambda^2(X),0,0,0) \colon X \in \so(7) \}.
\]
and set $\g{g}_0 = \g{k}_0 + \g{a}$.
Obviously, $\g{g}_0$ commutes with $\g{a}$.
Then $\g{g} = \g{g}_{-2\alpha} + \g{g}_{-\alpha} + \g{g}_0 + \g{g}_{\alpha} + \g{g}_{2\alpha}$ is the restricted root space decomposition of $\g{g}^*$ with respect to the maximal abelian subspace~$\g{a}$ of~$\g{p}^*$ and we have
\[
[H,\xi] = \beta(H) \xi
\]
for all $H \in \g{a}$ and $\xi \in \g{g}_{\beta}$, where $\alpha$ is the linear form on $\g{a}$ defined by $\alpha(0,0,e,0)=1$. Indeed, it follows immediately from the above calculations that
\begin{align*}
  [(0,0,e,0),(0,-\bar x,0,x)] &= (0, -\bar x,0,x), \\
  [(0,0,e,0),(0,\bar x,0,x)] &= -(0, \bar x,0,x)
\end{align*}
for $x \in \Ca$ and we compute, for $p \in \Pu(\Ca)$,
\begin{align*}
[(0,0,e,0),(\pm 2\lambda^2(e \wedge p),0,p,0)]
&= (4\lambda^2(e \wedge p),0,\mp2(e \wedge p)e,0) = \\
&=(4\lambda^2(e \wedge p),0,\mp2(ep^te-pe^te),0) =\\
&=\pm2(\pm2\lambda^2(e \wedge p),0,p,0);
\end{align*}
furthermore we have, for $X \in \so(7)$,
\[
[(0,0,e,0),(\lambda^2(X),0,0,0)]=(0,0,-\lambda(\lambda^2(X))(e),0) = (0,0,-X(e),0) = 0.
\]

It follows from $[\g{g}_\alpha, \g{g}_\beta] \subseteq \g{g}_{\alpha+\beta}$ that $\g{n} = \g{g}_\alpha + \g{g}_{2\alpha}$ is a nilpotent subalgebra of~$\g{g}^*$ and $\g{a} + \g{n}$ is a solvable Lie algebra of~$\g{g}^*$. Moreover,
\[
\g{g}^* = \g{k} + \g{a} + \g{n}
\]
is an Iwasawa decomposition and $\g{k}_0 + \g{a} + \g{n}$ is a maximal parabolic subalgebra of~$\g{g}^*$. It is well known that the closed subgroup~$AN$ of $F_4^*$ corresponding to $\g{a} + \g{n}$ acts simply transitively on $\M$, see e.g.~\cite{BTV95}.

Let us write down the bracket on $\g{k}_0 + \g{a} + \g{n}$ in an explicit form. We use the notation
\begin{align*}
\g{v}&:=\g{g}_{\alpha\ }=\{ (0,- \bar x,0,x)\in \g{g}^* \colon x \in \Ca \},\\
\g{z} &:= \g{g}_{2\alpha} = \{ (R_{p},0,p,0)\in \g{g}^* \colon p \in \Pu(\Ca) \}.
\end{align*}
Furthermore, identify $\g{a}$ with~$\R$ via the linear map that sends $(0,0,e,0)$ to~$1$.
Now we have the identification
\[
\g{k}_0 \times \g{a} \times \g{v} \times \g{z} = \so(7) \times \R \times \Ca \times \Pu(\Ca)
\]
and we may use the imbedding
\begin{align*}
\iota \colon \g{k}_0 \times \g{a} \times \g{v} \times \g{z} &\to \g{g}^*, \\ (A,s,x,p)&\mapsto(\lambda^2(A)+R_p,-\bar x,s+p,x).
\end{align*}
Then the restriction of the bracket on~$\g{g}^*$ to~$\g{k}_0 \times \g{a} \times \g{v} \times \g{z}$ is given by
\begin{equation}\label{eq:anbracket}
[\iota(A,s,x,p),\iota(B,t,y,q)] = \iota(AB-BA,0,z,r)
\end{equation}
where
\begin{align*}
z &= \lambda(A)y-\lambda(B)x+sy-tx,\\
r &= Aq-Bp+2sq-2tp+x\bar y -y \bar x.
\end{align*}


\section{Totally geodesic subspaces}\label{sec:totgeod}


The congruence classes of totally geodesic submanifolds of the Cayley hyperbolic plane are well known, see Proposition~\ref{prop:totgeod}. In this section will describe the corresponding Lie triple systems and the subalgebras generated by them in the framework of our model for the Lie algebra~$\g{g}^*$.

\begin{proposition}\label{prop:totgeod}
Every totally geodesic subspace of positive dimension in~$\M$ is congruent to one of $\R \HH^2,$ $\C \HH^2,$ $\H \HH^2$, $\M$, or $\HH^m$ for $1 \le m \le 8$, where we denote by $\R \HH^2$ and $\HH^2$ two non-congruent types of totally geodesic hyperbolic planes corresponding to the totally geodesic subspaces $\R \PP^2$ and $\eS^2$, respectively, in the compact dual $\Ca \PP^2$.
In particular, the maximal totally geodesic subspaces of~$\M$ are, up to congruence, $\H \HH^2$ and $\HH^8$.
\end{proposition}

\begin{proof}
See \cite[Proposition~3.1]{Ch73} or~\cite{Wo63}.
\end{proof}

In the following, we will write down some standardly embedded representatives of each congruence class of totally geodesic subspaces of~$\M$.  We will use the identification $\g{p}^* = \Ca \times \Ca$ given by $(0,0,x,y) \mapsto (x,y)$ as a convenient shorthand notation for elements in~$\g{p}^*$.

\begin{proposition}\label{prop:realhyp}
Let $V \subseteq \R^8$ be a non-zero linear subspace. Then $V \times \{0\}$ and $\{0\} \times V$ are both Lie triple systems in~$\g{p}^*$ whose exponential images are both congruent to $\HH^d$, where $d=\dim(V)$.
\end{proposition}

\begin{proof}
Consider the compact form $\g{g}$.  Applying the automorphism $\tau$ and using Lemma~\ref{lm:sonine}, it follows that the subspaces $\R^8 \times \{0\}$ and $\{0\} \times \R^8$ of~$\g{p}$ are both Lie triple systems generating a Lie algebra isomorphic to~$\so(9)$. It follows that they are Lie triple systems tangent to an $8$-sphere in~$\Ca\PP^2$.
Hence any of their $d$-dimensional linear subspaces is a Lie triple system tangent to a $d$-dimensional sphere in~$\Ca\PP^2$. Now the statement of the proposition follows by duality.
\end{proof}

The next proposition describes a Lie triple system tangent to a totally geodesic~$\H \HH^2$ in~$\M$ and a subalgebra of~$\g{g}^*$ isomorphic to~$\g{sp}(1,2)$, which is generated by this Lie triple system. To formulate it precisely, we need to make the following definitions.
Define $\H$ to be the subalgebra of~$\Ca$ spanned by $e_0=e,e_1,e_2,e_3$ and define the subalgebra
\begin{equation}\label{eq:so4def}
\so(4) :=
\left\{ \left(
\begin{array}{c|c}
  A &  \\ \hline
   & 0 \\
\end{array} \right)
\colon
A \in \R^{4 \times 4}, A^t = - A
\right\} \subset \so(8).
\end{equation}
Let $U = \{0\} \times \H \times \{0\} \times \{0\}$.
Then it follows from Lemma~\ref{lm:sonine} that $\so(4)+U$ is a subalgebra of~$\so(8)$ isomorphic to~$\so(5) \cong \mysp(2)$ and we have $[U,U]=\so(4)$.
Now consider the subspaces $\tau(U) = \{0\} \times \{0\} \times \H \times \{0\}$ and $\tau^2(U) = \{0\} \times \{0\} \times \{0\} \times \H$ of the compact Lie algebra~$\g{f}_4$. We have $[\tau(U),\tau(U)]=\lambda^2(\so(4))$ and $[\tau^2(U),\tau^2(U)]=\lambda(\so(4))$.
We define
\[
\mysp(1)^3 := \so(4) + \lambda(\so(4)) + \lambda^2(\so(4)).
\]
Let us show that this subspace is actually a subalgebra of~$\so(8)$: note that the representations $\lambda|_{\so(4)}$ and $\lambda^2|_{\so(4)}$ leave the subspaces $\H$ and $\H e_4$ of~$\Ca$ invariant and act nontrivially on both of them. Let $\varrho \colon \so(4) \to \GL(\H e_4)$ be the representation defined by $\varrho(A)(x) = \lambda(A)(x)$ and let $\varphi \colon \so(4) \to \GL(\H e_4)$ be the representation defined by $\varphi(A)(x) = \lambda^2(A)(x)$.
Since we have  $L_a|_{\H e_4} = -R_a|_{\H e_4}$ for $a \in \Pu(\H)$, it follows that the images of $\varrho$ and $\varphi$ are the same. Therefore, we have
\begin{equation}\label{eq:sp13descr}
\g{sp}(1)^3 =
\left\{ \left(
\begin{array}{c|c}
  A &  \\ \hline
   & \varrho(B) \\
\end{array} \right) \in \so(8)
\colon
A \in \R^{4 \times 4}, \; B \in \su(2)
\right\},
\end{equation}
where we have written $\su(2)$ for the simple ideal of~$\so(4)$ which does not lie in the kernel of~$\varrho$. It is now obvious that $\g{sp}(1)^3$ is a subalgebra of~$\so(8)$ isomorphic to $\mysp(1) \oplus \mysp(1) \oplus \mysp(1)$. Moreover, by definition, we have $\lambda(\g{sp}(1)^3) = \g{sp}(1)^3$.

\begin{proposition}\label{prop:quathyp}
The subset $\g{sp}(1)^3 \times \H \times \H \times \H$ is a subalgebra of~$\g{f}_4^*$ isomorphic to~$\g{sp}(1,2)$.
It is generated by the set $\H\times\H$, which is a Lie triple system in~$\g{p}^*$ whose exponential image is congruent to~$\H \HH^2$.
\end{proposition}

\begin{proof}
We first show that the subset $\g{sp}(1)^3 \times \H \times \H \times \H \subset \g{f}_4^*$ is closed under taking brackets. Obviously $\g{sp}(1)^3$ is a subalgebra of~$\g{f}_4^*$.
Since $\so(4) \subset \g{sp}(1)^3$, all brackets of elements in $\{0\}\times\H\times\{0\}\times\{0\}$ are contained in~$\g{sp}(1)^3$. Since $\g{sp}(1)^3$ is $\lambda$-invariant, it also contains all brackets of elements in $\{0\}\times\{0\}\times\H\times\{0\}$ and all brackets of elements in $\{0\}\times\{0\}\times\{0\}\times\H$.
Using~(\ref{eq:ncbracket}), it now follows easily that all brackets of elements in  $\{0\}\times\H\times\H\times\H$ are contained in~$\g{sp}(1)^3\times\H\times\H\times\H$.
Furthermore, it follows from the $\lambda$-invariance of~$\g{sp}(1)^3$ that $\ad_X$, $X \in \g{sp}(1)^3$ leaves the subspaces $\{0\}\times\H\times\{0\}\times\{0\}$, $\{0\}\times\{0\}\times\H\times\{0\}$, and $\{0\}\times\{0\}\times\{0\}\times\H$ invariant. This completes the proof of the subalgebra property.

It is now straightforward to see that $\{0\}\times\{0\}\times\H\times\H$ is a Lie triple system which generates the Lie algebra $\g{sp}(1)^3 \times \H \times \H \times \H$. Since this Lie algebra is 21-dimensional and the Lie algebra of the isometry group of an 8-dimensional totally geodesic subspace of~$\Ca\HH^2$, we know by Proposition~\ref{prop:totgeod} that it is isomorphic to~$\g{sp}(1,2)$ and the exponential image of $\{0\}\times\{0\}\times\H\times\H$  is congruent to $\H \HH^2$.
\end{proof}

From now on, we will use the notation
\[
\g{sp}(1,2) := \g{sp}(1)^3 \times \H \times \H \times \H.
\]

\begin{lemma}\label{lm:subalgtg}
Let $K$ be a subalgebra of~$\H$. Then $\ell := K \times K \subset \g{p}$ is a Lie triple system.
\end{lemma}

\begin{proof}
We compute some brackets:
\begin{align*}
[(a,0),(b,0)] &= (4\lambda^2(a \wedge b),0,0,0),\\
[(a,0),(0,b)] &= (0,-\overline{ab},0,0),\\
[(0,a),(b,0)] &= (0,\overline{ba},0,0),\\
[(0,a),(0,b)] &= (4\lambda(a \wedge b),0,0,0).
\end{align*}
The map $\ad_{(c,0)}$ sends these four elements to
\[
(0,0,-4(a \wedge b)c,0), \quad
(0,0,0,\bar c (ab)), \quad
(0,0,0,-\bar c (ba)), \quad
(0,0,-4\lambda^2(a \wedge b)c,0),
\]
while $\ad_{(0,c)}$ maps them to
\[
(0,0,0,-4\lambda(a \wedge b)c), \quad
(0,0,-(ab)\bar c,0), \quad
(0,0,(ba)\bar c,0), \quad
(0,0,0,-4(a \wedge b)c).
\]
We have $4\lambda(a \wedge b)c = 2L_{\bar b} \circ L_{\bar a} (c) = 2\bar b(\bar ac)$ and
$4\lambda^2(a \wedge b)c = 2R_{\bar b} \circ R_{\bar a} (c) = 2(c \bar a)\bar b$, where we have assumed $a \in \Pu(K)$. Furthermore, $(a \wedge b)c = ab^tc-ba^tc \in \spann_{\R}\{a,b\}$. We have shown that $[\ell,[\ell,\ell]]\subseteq \ell$.
\end{proof}

\begin{remark}\label{rem:hplanes}
To show that the Lie triple system given by Lemma~\ref{lm:subalgtg} for~$K = \R$ is not congruent to the Lie triple systems given by Proposition~\ref{prop:realhyp} for~$\dim(V)=2$, we consider their counterparts for the dual symmetric space~$\Ca \PP^2$.

Assume $\Ca \PP^2$ is endowed with the invariant metric induced by~$\langle{\cdot},{\cdot}\rangle$.
Let $X,Y \in \g{p}$ be a pair of orthonormal vectors.
Then the sectional curvature of the plane spanned by~$X$ and~$Y$ is given by
\[
K(X,Y) = \langle R(X,Y)X,Y \rangle
= \langle [[X,Y],X],Y \rangle
= \langle [X,Y],[X,Y] \rangle.
\]
For $X=(0,0,e,0)$, $Y=(0,0,0,e)$ we obtain
\[
K(X,Y) = \langle (0,e,0,0) ,(0,e,0,0) \rangle = 8.
\]
For $X=(0,0,e,0)$, $Z=(0,0,e_1,0)$ we obtain
\begin{align*}
K(X,Z) &= \langle (4\lambda^2(e_1 \wedge e),0,0,0) ,(4\lambda^2(e_1 \wedge e),0,0,0) \rangle = \\
&= -4 \tr (R_{\bar e} \circ R_{\bar e_1})^2 =-4 \tr (R_{\bar e_1})^2 =-4 \tr (-\id_{\R^8}) = 32.
\end{align*}
The two-dimensional Lie triple systems $\spann_{\R}\{X,Y\}$ and $\spann_{\R}\{X,Z\}$ either correspond to a totally geodesic~$\eS^2$ or to a totally geodesic~$\R\PP^2$ in~$\Ca \PP^2$. Since all geodesics in~$\Ca \PP^2$ are closed and of the same length, both spaces have the same diameter~$\delta$ and the totally geodesic~$\R\PP^2$ is covered by a locally isometric two-sphere of diameter~$2\delta$. This shows that the constant sectional curvature of a totally geodesic~$\eS^2$ in~$\Ca \PP^2$ is four times the constant sectional curvature of a totally geodesic $\R\PP^2$ in~$\Ca \PP^2$.

We have shown that the Lie triple system in~$\g{p}^*$ spanned by $(e,0)$, $(0,e)$ is tangent to a totally geodesic~$\R\HH^2$; the Lie triple system spanned by $(e,0)$, $(e_1,0)$ is tangent to a~$\HH^2$.

Moreover, these calculations show that the four-dimensional Lie triple system spanned by $(e,0)$, $(e_1,0)$, $(0,e)$, $(0,e_1)$ does not have constant sectional curvature, thus it is tangent to a~$\C \HH^2$. Therefore the Lie triple system given by Lemma~\ref{lm:subalgtg} for~$K = \C$ is not congruent to the Lie triple systems given by Proposition~\ref{prop:realhyp} for~$\dim(V)=4$. (This can also been seen by counting dimensions of the Lie algebras generated by the two Lie triple systems.)
\end{remark}

\begin{lemma}\label{lm:sumlietr}
Let $\ell \subseteq \g{p}^*$ be a Lie triple system such that $V_1 := \ell \cap (\Ca \times \{0\})$ and $V_2 := \ell \cap (\{0\} \times \Ca)$ are both non-zero. Then $\dim (V_1) = \dim(V_2)$ and $V_1+V_2 \subseteq \g{p}^*$ is a Lie triple system.
\end{lemma}

\begin{proof}
Let $W_1, W_2 \subseteq \Ca$ be such that $V_1 = W_1 \times \{0\}$ and $V_2 = \{0\} \times W_2$. Observe that $[\ell,\ell]$ contains the bracket $\xi := (0,-\overline{vw},0,0) = [(0,0,v,0),(0,0,0,w)]$ where $v \in W_1$, $w \in W_2$ are both non-zero. The linear map $\ad_\xi$ then induces isomorphisms $\Ca \times \{0\} \to \{0\} \times \Ca$ and $\{0\} \times \Ca \to \Ca \times \{0\}$ of real vector spaces; furthermore, viewed as a map $\g{p}^* \to \g{p}^*$, it preserves the subspace~$\ell$, since $\ell$ is a Lie triple system. It follows that $\ad_\xi$ maps~$V_1$ bijectively onto~$V_2$, hence $\dim (V_1) = \dim(V_2)$. This shows that all maps $\ad_\xi$ where $\xi = [(0,0,v,0),(0,0,0,w)]$, $v \in W_1$, $w \in W_2$, preserve $V_1+V_2$. Since, by Proposition~\ref{prop:realhyp}, $\Ca \times \{0\}$ and $\{0\} \times \Ca$ are Lie triple systems of~$\g{p}^*$, it follows that we also have $[[V_i,V_i],V_i] \subseteq V_i$ for $i=1,2$.
Furthermore, we have $[V_i,V_i] \subseteq \so(8)$ and hence $[[V_i,V_i],V_{3-i}] \subseteq V_{3-i}$.
Now the assertion of the lemma follows.
\end{proof}

\begin{remark}\label{rem:gtwoact}
In the proof of the next lemma we use the fact that ${\rm G}_2$ is a subgroup of the automorphism group of~$\g{g}^*$. In fact, let ${\rm G}_2$ be the set of automorphisms of~$\Ca$. Then ${\rm G}_2$ acts on~$\g{g}^*$ as follows: For $f \in {\rm G}_2$ and $(A,x,y,z) \in \so(8) \times \Ca^3$, define
\[
f \cdot (A,x,y,z) := (f \circ A \circ f^{-1},f(x),f(y),f(z)).
\]
Note that we have ${\rm G}_2 \subset \Spin(8)$ and the action defined above is given by the restriction of the $\Spin(8)$-action described in Lemma~\ref{lm:spin8aut}, see~\cite[\S3]{murakami}.
\end{remark}

\begin{lemma}\label{lm:totgeod}
Let $\ell \subseteq \g{p}^*$ be a Lie triple system which is of the form
$\ell = V_1 + V_2$, where $V_1 \subseteq \Ca \times \{0\}$ and $V_2 \subseteq \{0\} \times \Ca$ are non-zero linear subspaces. Then $\dim (V_1) = \dim(V_2)$ and the totally geodesic submanifold of~$\M$ corresponding to~$\ell$ is congruent to $\R\HH^2$, $\C \HH^2$, $\H \HH^2$, or~$\M$.
\end{lemma}

\begin{proof}
It follows from Lemma~\ref{lm:sumlietr} that $\dim (V_1) = \dim(V_2)$. Using the fact that $\Spin(8)$ acts transitively on $\eS^7 \times \eS^7$, we may assume that both $W_1$ and $W_2$, defined as in Lemma~\ref{lm:sumlietr}, contain~$e$.
It follows that $\Pu(W_i)=W_i$ and $\bar W_i = W_i$ for $i=1,2$.
Taking the bracket
\[
[(0,0,v,0),(0,0,e,0)] = (4\lambda^2(v \wedge e),0,0,0),
\]
we see that $[\ell,\ell]$ contains the elements $(\lambda^2(v \wedge e),0,0,0)$ for all $v \in W_1$. Setting $\eta := (2\lambda^2(v \wedge e),0,0,0)$, we obtain
\[
\ad_\eta(0,0,0,w) = (0,0,0,2\lambda(v \wedge e)w) ) = (0,0,0,L_{\bar v} (w)) = (0,0,0,\bar vw)\in V_2
\]
for all $v \in \Pu(W_1)$, $w \in W_2$. In particular, since $e \in W_2$, it follows that $W_1 \subseteq W_2$.

This shows that $W_1=W_2$ and furthermore, that $W_1$ is a subalgebra of~$\Ca$. Hence $W_1$ is isomorphic to one of $\R, \C, \H,$ or $\Ca$. In the last case, the assertion of the lemma is trivial. Otherwise, we may assume $W_1$ is a subalgebra of the standardly embedded $\H = \spann \{e, e_1, e_2, e_3\}$ after applying an automorphism of~$\Ca$. Indeed: if~$W_1 \cong \R$ and $W_1= \spann\{e\}$, there is nothing to prove; if $W_1 \cong \C$ and $W_1= \spann\{e,v\}$ for some unit vector $v \in \Pu(\Ca)$, there is an element~$f$ in the automorphism group~$\LG_2$ of~$\Ca$ such that $f(v)=e_1$, since $\LG_2$ acts transitively on the unit sphere in~$\Pu(\Ca)$; if $W_1 \cong \H$ and $W_1= \spann\{e,v, w, vw\}$ for unit vectors $v,w \in \Pu(\Ca)$ with $v \perp w$, there is an element~$f$ in the automorphism group~$\LG_2$ of~$\Ca$ such that $f(v)=e_1$, $f(w)=e_2$, and hence $f(vw)=e_1e_2=e_3$, since $\LG_2$ acts transitively on the Stiefel manifold of orthonormal two-frames in~$\Pu(\Ca)$, cf.~\cite{oniscik}. Now the lemma follows from Proposition~\ref{prop:quathyp} and Remark~\ref{rem:hplanes}.
\end{proof}


\section{Polar actions}\label{sec:polact}


The following criterion for an isometric action on~$M$ to be polar was proved in~\cite[Proposition~2.3]{DDK12} for an arbitrary Riemannian symmetric space of non-compact type. Note that sections of polar actions are always totally geodesic submanifolds.

\begin{proposition}\label{prop:criterion}
Let $\Sigma$ be a connected totally geodesic submanifold of~$\M$ with $o \in
\Sigma$. Let $H$ be a closed subgroup of~$I(\M)$. Then $H$ acts polarly on~$\M$
with section~$\Sigma$ if and only if $T_o\Sigma$ is a section of the slice
representation of~ $H_o$ on~$N_o(H \cdot o)$, and $\langle
\g{h},[T_o\Sigma,T_o\Sigma]\rangle=0$.
\end{proposition}

Isometric Lie group actions of cohomogeneity one (i.e.\ the regular orbits are hypersurfaces) are a special case of polar actions of independent interest. They have been classified by Hsiang-Lawson~\cite{hsl} on spheres, on complex projective space by Takagi~\cite{takagi}, on quaternionic projective space by D'Atri~\cite{datri}, and on the Cayley projective plane by~Iwata~\cite{iwata}. In~\cite{k02}, the author classified all cohomogeneity one actions on the remaining compact irreducible Riemannian symmetric spaces. For classification results concerning isometric cohomogeneity-one actions on non-compact Riemannian symmetric spaces see~\cite{BT04,BT07} and the references therein.

Assume there are Riemannian manifolds~$X,\,Y$ and Lie groups~$G,\,H$ such that $G$ acts isometrically on~$X$ and $H$ acts isometrically on~$Y$. Then we say the $G$-action on~$X$ and the $H$-action on~$Y$
are \emph{orbit equivalent} if there is an isometry $f \colon X \to Y$ such that $f$ maps each connected component of a $G$-orbit in~$X$ to a connected component of an $H$-orbit in~$Y$.

\begin{example}\label{ex:horo}
Consider the action of the nilpotent group~$N$ on~$\M$. Since
the group $AN$ acts simply transitive on~$\M$ and $\dim(N)=\dim(AN)-1$, the $N$-action is of cohomogeneity one. Indeed, the orbits of this action are the leaves of the well-known horospherical foliation on~$\M$, where each orbit is a horosphere.
\end{example}

\begin{example}\label{ex:coh1A}
Let $\g{m} \subset \g{v}$ be a linear subspace of codimension one. Then $\g{h} = \g{a} + \g{m} + \g{z}$ is a subalgebra of $\g{a} + \g{n}$ of codimension one and the closed connected subgroup of~$\G$ with Lie algebra~$\g{h}$ acts with cohomogeneity one on~$\M$. Since there are no singular orbits, the orbits of this action are the leaves of a regular foliation of codimension one.
\end{example}

By the results of~\cite{BT03}, the orbit foliations given in the Examples~\ref{ex:horo} and~\ref{ex:coh1A} above exhaust all orbit equivalence classes of isometric cohomogeneity one actions without singular orbits on~$\M$. In \cite{BT04}, isometric cohomogeneity one actions on irreducible non-compact symmetric spaces with a totally geodesic singular orbit were classified.

\begin{example}\label{ex:contg}
Let $\g{m} \subset \g{v}$ be a linear subspace of codimension~$d$, let \[N_{\g{k}_0}(\g{m}) = \{ X \in \g{k}_0 \colon [X,\g{m}] \subseteq \g{m} \}\] be the normalizer of~$\g{m}$ in~$\g{k}_0$ and let $N_{K_0}(\g{m})$ be the normalizer of~$\g{m}$ in~$K_0$. Then $\g{h} = N_{\g{k}_0}(\g{m}) + \g{a} + \g{m} + \g{z}$ is a subalgebra of $\maxpar$. It follows from~\cite{oniscik}, see also~\cite[p.~3435\textit{f}]{BT07}, that $N_{K_0}(\g{m})= \LG_2$ if $d=1$ or~$7$, that $N_{K_0}(\g{m})= \U(3)$ if $d=2$ or~$6$, and that $N_{K_0}(\g{m})= \SO(4)$ if $d=3$ or~$5$. Furthermore, it is easily seen that: in case $d=1$, the action is orbit equivalent to Example~\ref{ex:coh1A}, in the cases where $d=2,3,6,7$ the slice representation of the $H$-action at~$o$ is of cohomogeneity one.
It is also shown in~\cite{BT07} that if $d=4$, the slice representation is of cohomogeneity one and there is a one-parameter family of orbit equivalence classes of these actions.

If the subspace $\g{m}$ is trivial, then $N_{K_0}(\g{m}) = \Spin(7)$ and $H$ acts also with cohomogeneity one, however, the $H$-orbit is a totally geodesic $\HH^8$ and the $H$-action on~$\M$ is orbit equivalent to the $\Spin(1,8)$-action, see Section~\ref{sec:totgorb}.

Thus the corresponding closed connected subgroup $H$ of~$\G$ with Lie algebra~$\g{h}$ acts with cohomogeneity one on~$\M$ if $d \in \{1,2,3,4,6,7,8\}$.
\end{example}

It is shown in~\cite{BT07} that the actions in Examples~\ref{ex:horo}, \ref{ex:coh1A}, and \ref{ex:contg} and the actions of~$\Spin(9)$, $\Sp(1) \cdot \Sp(1,2)$ and $\Spin(1,8)$ exhaust all orbit equivalence  classes of cohomogeneity one actions on~$\M$. Indeed, up to orbit equivalence, the actions in Examples~\ref{ex:contg} with $d \in \{2,3,4,6,7\}$ are exactly the actions of cohomogeneity one with a non-totally geodesic singular orbit.

As an immediate application of the criterion in Proposition~\ref{prop:criterion}, we show that a certain regular homogeneous foliation of~$\M$ is polar.

\begin{example}\label{ex:coh2nilp}
Consider the subalgebra $\g{h} = \g{m} + \g{z}$ of~$\g{n}$, where $\g{m} = \{(0,x,0,x) \colon x \in \Pu(\Ca)\}$. Let $H$ be the closed connected subgroup of~$\G$ whose Lie algebra is~$\g{h}$. This group acts with cohomogeneity two on~$\M$. We will show that the action is polar.
The normal space $N_o(H \cdot o) \subset \g{p}^*$ is spanned by the vectors $(e,0)$, $(0,e)$ and it follows that $[T_o\Sigma,T_o\Sigma]$ is spanned by the vector $[(e,0), (0,e)] = (0,-e,0,0)$, which is orthogonal to~$\g{h}$. Thus it follows by Proposition~\ref{prop:criterion} that the action is polar. It follows from Lemma~\ref{lm:totgeod} that the section is a totally geodesic $\R \HH^2$.  Note that this action is not orbit equivalent to any of the actions in Table~\ref{tbl:actions} since it has no singular orbits.
\end{example}

\begin{lemma}\label{lm:nochsec}
A polar action on~$\M$ with a section congruent to~$\C \HH^2$ or~$\H \HH^2$ has no singular orbits.
\end{lemma}

\begin{proof}
Assume there is a singular orbit. Then this singular orbit contains a point~$p \in \Sigma$ and the slice representation of~$H_p$ on~$N_p(H \cdot p)$ is a polar representation with section~$T_p\Sigma$ and orbits of positive dimension. It follows that the generalized Weyl group~$W(T_p\Sigma)$ of the $H_p$-action on~$N_p(H \cdot p)$ contains an element~$w$ which acts on~$T_p\Sigma$ as a reflection. This implies that $\Sigma$ contains the totally geodesic hypersurface which is given by the connected component of the fixed point set of~$w$ on~$\Sigma$. Hence we have arrived at a contradiction since neither $\C \HH^2$ nor~$\H \HH^2$ contains totally geodesic hypersurfaces, cf.~\cite{Wo63}.
\end{proof}


\section{Classification of polar actions with an invariant totally geodesic subspace}\label{sec:totgorb}


In this section we will study polar actions on~$\M$. Since it has been a successful strategy in the special case of cohomogeneity one actions to start with actions which have a totally geodesic singular orbit~\cite{BT04}, we will also proceed along these lines.
It turns out that this approach can be refined to include actions that leave a totally geodesic subspace invariant which is not necessarily an orbit.

In~\cite[Theorem~8.5]{Ch73}, the following classification of closed connected subgroups in~$\G$ was obtained.

\begin{theorem}[Chen~1973]\label{th:subgroups}
Let $H$ be a connected Lie subgroup of~$\G$, then $H$ is a conjugate of the
following:
\begin{enumerate}

\item a subgroup of~$\Spin(9)$,

\item a subgroup of the invariant group of a boundary point,

\item $\Spin(1,m) \cdot L$, where $L \subseteq \Spin(8-m)$, $2 \le m \le 8$,

\item $\SO_0(1,2) \cdot L$, where $L \subseteq \mathrm{G}_2$,

\item $\SU(1,2) \cdot L$, where $L \subseteq \SU(3)$,

\item $\Sp(1,2) \cdot L$, where $L \subseteq \Sp(1)$,

\item $\G$.

\end{enumerate}
\end{theorem}

This theorem implies, together with the following proposition, that the proper closed subgroups of~$\G$ which do not leave a totally geodesic subspace invariant, are as described in item~(ii) of Theorem~\ref{th:subgroups}.

\begin{proposition}\label{prop:tgnorm}
Let $P$ be a totally geodesic subspace of~$\M$. Let \[N(P) = \{ g \in \G \colon g\cdot P = P \}\] be its normalizer and let \[Z(P) = \{ g \in \G \colon g \cdot p = p \;\forall p \in P \}\] be its centralizer in~$\G$.
Then their connected components $Z(P)_0$ and $N(P)_0$ are given (up to automorphisms of~$\G$) in Table~\ref{tbl:actions}. Furthermore, for each action of~$N(P)_0$ on~$\M$, a non-trivial slice representation, the cohomogeneity, and the fact whether it is polar or not, is stated.

\begin{table}[h]
\begin{minipage}{\textwidth}
\[\begin{array}{|c|c|c|c|c|c|}
\hline
P & Z(P)_0 & N(P)_0 & \textsl{slice representation} & \textsl{cohom} & \textsl{polar?} \\

\hline

\{\mathrm{pt}\} & \Spin(9) & \Spin(9) & (\Spin(9),\R^{16}) & 1 & \textrm{yes} \\ 

 \R \HH^2 & \LG_2 & \LG_2 {\cdot} \SO_0(1,2) & (\LG_2\times\O(2),\Pu(\Ca) \otimes_\R \R^2) & 2 & \textrm{yes} \\ 

 \C \HH^2 & \SU(3) & \SU(3) {\cdot} \SU(1,2) & (\SU(3)\times\U(2),\C^3 \otimes_\C \C^2) & 2 & \textrm{yes} \\ 

 \H \HH^2 & \Sp(1) & \Sp(1) {\cdot} \Sp(1,2) & (\Sp(1)\times\Sp(2),\H^1 \otimes_\H \H^2) & 1 & \textrm{yes} \\ 

 \HH^1 & \Spin(7) & \Spin(7) {\cdot} \SO_0(1,1) & (\Spin(7),\R^7 \oplus \R^8) & 2 & \textrm{yes} \\ 

 \HH^2 & \SU(4) & \Spin(6){\cdot} \Spin(1,2) & (\U(4),\R^6 \oplus \C^4) & 2 & \textrm{yes} \\ 

 \HH^3 & \Sp(2) & \Spin(5) {\cdot} \Spin(1,3) & (\Sp(1){\cdot}\Sp(2),\R^5\oplus\H\otimes_{\mathbb H}\H^2) &  3 &  \textrm{no} \\ 

 \HH^4 & \Sp(1) {\cdot} \Sp(1) & \Spin(4) {\cdot} \Spin(1,4) & (\Sp(1)^4,\H \oplus \H \oplus \H) & 3 & \textrm{no} \\ 

 \HH^5 & \Sp(1) & \Spin(3)  {\cdot} \Spin(1,5) & (\Sp(1){\cdot}\Sp(2),\R^3 \oplus \H^2) & 2 & \textrm{yes} \\ 

 \HH^6 & \SO(2) & \SO(2) {\cdot} \Spin(1,6) & (\U(4), \R^2 \oplus \C^4) & 2 & \textrm{yes} \\ 

 \HH^7 & \{1\} & \Spin(1,7) & (\Spin(7),\R\oplus\R^8) & 2 & \textrm{yes} \\ 

 \HH^8 & \{1\} & \Spin(1,8) & (\Spin(8),\R^8) & 1 & \textrm{yes} \\ \hline

\end{array}\]
\begin{center}
\hfill\\
\caption{Isometric actions on~$\Ca\HH^2$ by connected normalizers of totally geodesic submanifolds.}\label{tbl:actions}
\hfill\\
\end{center}
\end{minipage}
\end{table}
\end{proposition}

\begin{proof}
Follows by duality~\cite{k11} from the Table in~\cite[Addendum]{gk16}.
\end{proof}

\begin{remark}\label{rem:actions}
In Table~\ref{tbl:actions} the information about a slice representation at a point~$p$ in the totally geodesic orbit~$P$ is given in the following form: a pair~$(G,V)$ for every action such that the linear action of~$G$ on~$V$ is equivalent to the effectivized slice representation restricted to the connected component $(N(P)_p)_0$ of the isotropy subgroup at~$p$.
The slice representation of the $\Spin(4) \cdot \Spin(1,4)$-action is equivalent to the action of~$\Sp(1)^4$ on~$\H^3$ given by $(a,b,c,d) \cdot (x,y,z) = (axb^{-1},ayc^{-1},bzd^{-1});$ the equivalence classes of the other slice representations are clear from the table.
\end{remark}

We will now carry out the classification of polar actions on~$\M$ leaving a totally geodesic subspace invariant.
Note that slice representations of polar actions are polar, cf.~\cite[Thm.~4.6]{PaTe}; more precisely, if there is a section~$\Sigma$ of a polar action on a Riemannian manifold and we have $p \in \Sigma$, then a section for the slice representation at~$p$ is given by~$T_p\Sigma$.

\subsection{\texorpdfstring{Actions with an invariant totally geodesic subspace~$\HH^8$}{Actions with an invariant totally geodesic subspace H8}}
Assume that the closed subgroup~$H \subseteq \G$ acts polarly on~$\M$ and in such a fashion that the totally geodesic subspace~$\HH^8 = \exp_o(\Ca \times \{0\})$ is invariant. We start with the case
where $\HH^8$ is an orbit of the $H$-action.

\subsubsection{\texorpdfstring{Actions with a totally geodesic orbit~$\HH^8$}{Actions with a totally geodesic orbit H8}}

The subalgebra $\so(8) \times \{0\} \times \Ca \times \{0\}$ is isomorphic to $\so(1,8)$ and it is the Lie algebra of the connected subgroup~$\Spin(1,8)$ of~$\G$. This subgroup acts on~$\M$ with cohomogeneity one; the action has a unique singular orbit, which is the exponential image of~$\Ca \times \{0\} \subset \g{p}^*$. We denote this orbit by~$\HH^8$. Assume $H \subset \Spin(1,8)$ is a closed connected subgroup acting polarly on~$\M$ such that the orbit through~$o$ is also~$\HH^8$.

By the classification of groups acting transitively on hyperbolic spaces given in~\cite[Theorem~6]{galaev} or \cite[Thm.~3.1]{cls}, we know that $H$ is conjugate to a group of the form $(A^\Phi \times H_o) \ltimes Z$, where $H_o \subset \Spin(7)$ is a subgroup, $\Phi \colon A \to \Spin(7)$ is a homomorphism (which may be trivial) and
\[
A^\Phi = \{\Phi(a) \cdot a \colon a \in A\}.
\]
Thus we may assume $\g{h}\supseteq\g{a}^\Phi \ltimes \g{z}$.
We may furthermore assume that the $H$-action is not orbit equivalent to the $\Spin(1,8)$-action. Then the slice representation of~$H_o$ on~$\{0\} \times \Ca = T_o(H \cdot o)$ is of cohomogeneity at least two. Since the slice representation is polar, it has a section $T_o\Sigma$, where $\Sigma$ is a section of the $H$-action on~$\M$, such that~$(0,e) \in T_o\Sigma$. Since the slice representation is at least of cohomogeneity two, there is a unit vector $v \in \Pu(\Ca)$ such that $(0,v) \in T_o\Sigma$. After replacing~$H$ by the conjugate group~$gHg^{-1}$ for a suitable element~$g \in \LG_2$, see Remark~\ref{rem:gtwoact}, we may assume $v=e_1$.
In particular, by~(\ref{eq:ncbracket}) we have
\[
(4\lambda (e \wedge e_1),0,0,0) = (-2L_{\bar e_1},0,0,0)  = (2L_{e_1},0,0,0) \in [T_o\Sigma,T_o\Sigma].
\]
On the other hand, we have that $\Ad_g(\g{z}) = \g{z}$, in particular, $\g{h}$~contains the subalgebra
\[
\g{z} = \{(R_{y},0,y,0) \colon y \in \Pu(\Ca) \}
\]
and hence the element $(R_{e_1},0,e_1,0)$. An explicit calculation shows that
\[
\langle (L_{e_1},0,0,0), (R_{e_1},0,e_1,0) \rangle = \tr( L_{e_1} R_{e_1} ) = 4.
\]
Hence the $H$-action on~$\M$ is not polar by Proposition~\ref{prop:criterion}.

\subsubsection{\texorpdfstring{Actions with an invariant totally geodesic subspace~$\HH^8$ which is not an orbit}{Actions with an invariant totally geodesic subspace H8 which is not an orbit}}
Assume the action of a closed connected subgroup~$H \subset \G$ on~$\M$ leaves a totally geodesic subspace~$\HH^8$ invariant, but acts non-transitively on~$\HH^8$. Then it follows that the $H$-action induces a polar action on~$\HH^8$ by~\cite[Lemma~4.2]{k07}. By the results of~\cite{wu}, see also~\cite[p.~328]{D12}, any polar action on~$\HH^8$ leaves a totally geodesic~$\HH^k$, $k \in \{1,\dots,7\}$, or a point, invariant. Thus, the action will be treated below.

\subsection{\texorpdfstring{Actions with an invariant totally geodesic subspace~$\HH^k$, $k=1,\dots,7$}{Actions with an invariant totally geodesic subspace Hk, k=1...7}}\label{subsec:hinv}
Let $H$ be a connected closed subgroup of one of the groups~$N(P)_0$ in Table~\ref{tbl:actions}. Assume that $H$ acts polarly on~$\M$. We may furthermore assume that the $H$-action is not orbit equivalent to the $N(P)_0$-action; hence it is of cohomogeneity greater than two.
Note that the normal space $N_oP \subset \g{p}^*$ is given by $U \times \Ca$, where $U \subset \Ca$ is a linear subspace of dimension~$8-k$ and in all cases the slice representation of the $N(P)_0$-action at the point~$o$ is reducible with the two invariant subspaces $U \times \{0\}$ and $\{0\} \times \Ca$, see Table~\ref{tbl:actions}.

These are also invariant subspaces of the slice representation of~$H_o$ and it follows from \cite[Thm.~4]{Da85} that a section of the polar slice representation of~$H_o$ is of the form $V_1 \oplus V_2$, where $V_1 \subset \Ca \times \{0\}$ and $V_2 \subset \{0\} \times \Ca$.

Now it follows from Lemma~\ref{lm:totgeod} that the section is either a~$\C \HH^2$ or an~$\H \HH^2$ and by Lemma~\ref{lm:nochsec} we know that the action of $H$ on~$\M$ has no singular orbits. But there is an invariant subspace of dimension~$k \le 7$ and thus the cohomogeneity is at least~$9$. Thus the section of this polar action is at least of dimension~$9$, leading to a contradiction, since sections of polar actions are totally geodesic, however, by Proposition~\ref{prop:totgeod} there are no totally geodesic submanifolds of~$\M$ of dimension greater than~$8$, except $\M$ itself, and it follows that the $H$-action is trivial.

\subsection{\texorpdfstring{Actions with an invariant totally geodesic subspace~$\H\HH^2$}{Actions with an invariant totally geodesic subspace HH2}}

Let $\Sp(1,2)$ be the closed connected subgroup of~$\G$ whose Lie algebra is as described in Proposition~\ref{prop:quathyp}. This group acts with cohomogeneity one on~$\M$ and such that a totally geodesic~$\H \HH^2$ containing~$o$ is the unique singular orbit of this action~\cite{PT99,k11}.

It is obvious from (\ref{eq:sp13descr}) that there is a subalgebra~$\g{c}$ isomorphic to~$\g{sp}(1)$ which centralizes $\g{sp}(1)^3$ in~$\so(8)$ such that the sum of~$\g{c}$ and $\g{sp}(1)^3$ is given by
\begin{equation}\label{eq:so4so4def}
\g{sp}(1)^4 :=
\left\{ \left(
\begin{array}{c|c}
  A &  \\ \hline
   & B \\
\end{array} \right) \in \so(8)
\colon
A,B \in \R^{4 \times 4}
\right\}.
\end{equation}
It is not hard to see that~$\g{c}$ in fact centralizes all of~$\g{sp}(1,2)$: indeed, consider the subgroup of~$\Aut(\so(8))$ generated by~$\lambda$; it is of order three and acts by cyclic permutation on the simple ideals of~$\g{sp}(1)^3$. On the other hand, one can directly verify that $\lambda(\g{sp}(1)^4) = \g{sp}(1)^4$, hence we have $\tau(\g{c}) = \lambda(\g{c}) = \g{c}$.  It follows that~$\g{c}$ centralizes $\{0\} \times \H \times \H \times \H$.
We have shown that $\g{sp}(1)^4 \times \H \times \H \times \H$ is a subalgebra of~$\g{f}_4$ isomorphic to~$\g{sp}(1)\oplus \g{sp}(1,2)$. The action of the corresponding connected closed subgroup~$L := \Sp(1) \cdot \Sp(1,2)$ of~$\G$ on~$\M$ is orbit equivalent to the~$\Sp(1,2)$-action~\cite{PT99,k11}.

Let $H \subseteq L$ be a closed connected subgroup which acts polarly on~$\M$. We may assume the $H$-action is not orbit equivalent to the $L$-action.

Consider the slice representation of the $L$-action at~$o$, which is equivalent to the restriction of the linear action of~$\Sp(1)^2 \cdot \Sp(2)$ on~$\H^1 \otimes_{\H} \H^2 = e_4 \H \times e_4\H \subset \g{p}^*$, where one of the two $\Sp(1)$-factors acts trivially.

Assume that the restriction of this representation to the isotropy group~$H_o$ of the $H$-action is of cohomogeneity one. By the classification of Lie groups acting transitively on spheres, it then follows that $H_o$ contains $\Sp(2)$ and it follows that $H_o$ also acts irreducibly on the tangent space $T_o(L \cdot o) = \H \times \H \subset \g{p}^*$ of the $L$-action. Thus the $H$-orbit through~$o$ is either equal to the $L$-orbit through~$o$ or a point. In the former case it follows that the $H$-action and the $L$-action are orbit equivalent; in the latter case, the action has a fixed point and will be treated below.

We may now assume that the slice representation of the $H$-action at~$o$ is not of cohomogeneity one. It follows that the connected component~$(H_o)_0$ of the isotropy subgroup
~$H_o$ is contained in a subgroup of~$\Sp(1)^2 \cdot \Sp(2)$ of the form $\Sp(1)^2 \cdot Q$, where~$Q$ is a maximal connected subgroup of~$\Sp(2)$. According to~\cite{dynkin1,dynkin2}, the maximal connected subgroups of~$\Sp(2)$ are, up to conjugation,
\[
\Sp(1) \times \Sp(1), \quad \U(2), \quad S,
\]
where the first group acts reducibly on~$\H^2 = \H^1 \oplus \H^1$, the second group acts on~$\H^2 = \C^2 \oplus \C^2$ by two copies of the standard $\U(2)$-representation and $S$~is a three-dimensional principal subgroup of~$\Sp(2)$ given by the $4$-dimensional irreducible representation of~$\SU(2)$, cf.~\cite{dynkin1}. We will now examine these three possibilities case by case:

\begin{enumerate}

\item Assume $(H_o)_0$ is contained in $\Sp(1)^2 \cdot (\Sp(1) \times \Sp(1))$.  Then the slice representation is polar and reducible.  By conjugating~$H$ with a suitable isometry of~$\M$ we may then furthermore assume that $\H e_4 \times \{0\}$ and $\{0\} \times \H e_4$ are invariant subspaces of the slice representation. It now follows from~\cite[Theorem~4]{Da85} that the section of the polar $H$-action is as described in Lemma~\ref{lm:totgeod} and hence congruent to $\R \HH^2$, $\C \HH^2$, or~$\H \HH^2$. In all three cases, it follows that all vectors of the form $(u e_4, v e_4)$, $u,v \in \H$, are contained in $T_o\Sigma \subset \g{p}^*$ for some section~$\Sigma$ of the polar $H$-action on~$\M$. By the criterion in Proposition~\ref{prop:criterion} it now follows that $\{0\} \times \H \times \{0\} \times \{0\} \perp \g{h}$ and hence that the Lie algebra~$\g{h}$ is contained in~$\g{sp}(1)^4 \times \{0\} \times \H \times \H$.

    Now assume there are two elements $X=(A,0,x,0)$, $Y=(B,0,0,y) \in \g{h}$, where $x,y \in \H \setminus \{0\}$. Their bracket is $[X,Y] = (AB-BA,-\overline{xy},-\lambda(B)x,\lambda^2(A)y)$, which is not contained in~$\g{h}$, a contradiction. Thus it follows that $\g{h}$ is actually a subalgebra of~$\g{sp}(1)^4 \times \{0\} \times \H \times \{0\}$ or~$\g{sp}(1)^4 \times \{0\} \times \{0\} \times \H$. Hence $H$ is conjugate to a subgroup of~$\Spin(4) \cdot \Spin(1,4)$ and the $H$-action was already treated in~Subsection~\ref{subsec:hinv}.

\item In the case where $(H_o)_0$ is contained in $\Sp(1)^2 \cdot \U(2)$, the slice representation is the direct sum of two equivalent modules. Since we have treated the case of a slice representation with zero-dimensional orbits already implicitly in part~(i), we may assume that the orbits of the slice representation have positive dimension and thus it follows from~\cite[Lemma~2.9]{k02} that the slice representation is non-polar.

\item Now assume $(H_o)_0 \subseteq \Sp(1)^2 \cdot S$. Consider the slice representation of the $L$-action on~$\M$, restricted to the subgroup~$\Sp(1)^2 \cdot S$ of~$L_o$. After dividing out the effectivity kernel, it is equivalent to the isotropy representation of the Riemannian symmetric space~$\mathrm{G}_2/\SO(4)$, see~\cite[Appendix]{k02}, in particular, it is irreducible and of cohomogeneity two. It follows from~\cite[Thm.~6]{KP} that this representation restricted to~$(H_o)_0$ is either trivial or orbit equivalent to the action of~$\Sp(1)^2 \cdot S$. We may rule out the former case as above. In the latter case it follows that $H$~contains~$S$. Now it follows from the fact that $\Sp(2) / S$ is a strongly isotropy irreducible homogeneous space that
    \[
    \left. \Ad_{\Sp(1,2)} \right|_S = \Ad_S \oplus U \oplus V \oplus \g{sp}(1)
    \]
    where $U$ is a $7$-dimensional irreducible module and $V = \H^2 = \H \times \H \subset \g{p}^*$ is also irreducible. Since $U$ and $V$ are irreducible $S$-modules, we have that the projection of~$\g{h}$ on~$\g{sp}(1,2)$ either contains~$V$ or is contained in $\g{s}+U+\g{sp}(1)$, where we denote the Lie algebra of~$S$ by~$\g{s}$. In the former case, since the Lie triple system~$V$ generates~$\g{sp}(1,2)$, the $H$-action is orbit equivalent to the~$L$-action; in the latter case it has a fixed point and will be treated below.

\end{enumerate}

\subsection{\texorpdfstring{Actions with an invariant totally geodesic subspace~$\C\HH^2$ or~$\R\HH^2$}{Actions with an invariant totally geodesic subspace CH2 or RH2}}

\subsubsection{\texorpdfstring{Actions with a totally geodesic orbit~$\C\HH^2$ or~$\R\HH^2$}{Actions with a totally geodesic orbit CH2 or RH2}}
The action of~$\SU(3)\cdot\SU(1,2)$ on~$\M$ is polar and of cohomogeneity two with an irreducible slice representation, its orbit through~$o$ is a totally geodesic~$\C\HH^2$, cf.~Table~\ref{tbl:actions}. Assume there is a closed subgroup~$H \subseteq \SU(3)\cdot\SU(1,2)$ whose action on~$\M$ is polar, has the totally geodesic orbit~$H \cdot o = \C\HH^2$ and is not orbit equivalent to the action of~$\SU(3)\cdot\SU(1,2)$. Then the slice representation of~$H_o$ on~$N_o\C\HH^2$, given by the restriction of the slice representation of the $\SU(3)\cdot\SU(1,2)$-action, is of cohomogeneity greater than two and it follows from~\cite[Theorem~6]{KP} that the slice representation of the $H$-action at~$o$ is trivial. This leads to a contradiction, since a section of this action would be a $12$-dimensional totally geodesic submanifold of~$\M$.

The case of the action of $\LG_2 \cdot \SO_0(1,2)$ can be treated in an analogous way, since it is also of cohomogeneity two and has an irreducible slice representation at a point in a totally geodesic orbit congruent  to~$\R\HH^2$.

\subsubsection{\texorpdfstring{Actions with an invariant totally geodesic subspace~$\C\HH^2$ or~$\R\HH^2$ which is not an orbit}{Actions with an invariant totally geodesic subspace CH2 or RH2 which is not an orbit}}
Now let $L := \SU(3)\cdot\SU(1,2)$ and consider a polar action of a closed connected subgroup~$H$ of~$L$.

We may assume that $T_o(L \cdot o) = \spann\{(e,0),(e_1,0),(0,e),(0,e_1)\}$. Since $T_o(H \cdot o)$ is a proper subspace of~$T_o(L \cdot o)$, there is a unit vector $v \in T_o(L \cdot o)$ in the normal space $N_o(H \cdot o)$.
Using the fact that $\C\HH^2$~is a totally geodesic subspace in~$\M$ and a symmetric space of rank one, we may assume that $v = (e,0)$.

The slice representation of the $H$-action at~$o$ has the two invariant subspaces $N_o(H \cdot o) \cap T_o(L \cdot o)$ and $N_o(H \cdot o) \cap N_o(L \cdot o)$. It follows from~\cite[Thm.~4]{Da85} that the tangent space~$T_o\Sigma$ of a section through~$o$ contains~$v$ and also a vector from~$\{0\} \times \Ca$.
Since the cohomogeneity of the $H$-action is at least three, it now follows from Lemma~\ref{lm:totgeod} that a section for the $H$-action is a $\C\HH^2$ or~$\H\HH^2$. Thus Lemma~\ref{lm:nochsec} implies there are no singular orbits. However, since one of the orbits is contained in a $\C\HH^2$, the cohomogeneity is~$12$ and it follows that $H = \{e\}$.

An analogous argument works for subgroups of~$\LG_2 \cdot \SO_0(1,2)$ on~$\M$.

\subsection{Actions with a fixed point}
Actions with a fixed point on~$\M$ have been classified in~\cite{DK10}. If a connected closed subgroup~$H$ of~$\G$ acts polarly and with a fixed point on~$\M$, then $H$ is conjugate to $\Spin(9)$ or one of its subgroups~$\Spin(8)$,
$\Spin(7) \cdot \SO(2)$, $\Spin(6) \cdot \Spin(3)$.


\section{Some concluding remarks}\label{sec:concl}


Let us mention some things that are left to explore in connection with this paper. Can one find similar octonionic formulae for the Lie algebras of types~$\mathrm{E}_6$, $\mathrm{E}_7$, $\mathrm{E}_8$? It also remains to classify those polar actions on~$\M$ which do not leave any totally geodesic subspace invariant. By Theorem~\ref{th:subgroups}, it suffices to study subgroups whose Lie algebras are contained in the maximal parabolic subalgebra~(\ref{eq:anbracket}).



\begin{thebibliography}{99}


\bibitem{adams} J.F. Adams, \emph{Lectures on exceptional Lie groups.} University of Chicago Press (1996).

\bibitem{Baez02} J.C.~Baez, \emph{The octonions.} \emph{Bull.\ Amer.\ Math.\ Soc.}\ \textbf{39}(2) (2002) 145--205; errata: \emph{Bull.\ Amer.\ Math.\ Soc.} \textbf{42}, (2005) 213--213.

\bibitem{B98} J.~Berndt, \emph{Homogeneous hypersurfaces in hyperbolic spaces.} Math.\ Z., \textbf{229}(4), (1998) 589--600.

\bibitem{B11} J.~Berndt, \emph{Polar actions on symmetric spaces.} In Workshop on Diff. Geom.~\textbf{15}, (2011) 1--10.

\bibitem{BB01} J.~Berndt, M.~Br\"{u}ck, \emph{Cohomogeneity one actions on hyperbolic spaces}, \emph{J.\ Reine Angew.\ Math.}\ \textbf{541}, (2001) 209--235.

\bibitem{BDT10} J.\ Berndt, J.C.\ D\'{\i}az-Ramos, H.\ Tamaru, \emph{Hyperpolar
    homogeneous foliations on symmetric spaces of noncompact type}, \emph{J.\
    Differential Geom.}\ \textbf{86}(2), (2010) 191--235.

\bibitem{BT03} J.~Berndt, H.~Tamaru, \emph{Homogeneous codimension one foliations on noncompact symmetric spaces}, \emph{J.~Differential Geom.}\ \textbf{63}(1),(2003)  1--40.

\bibitem{BT04} J.~Berndt, H.~Tamaru, \emph{Cohomogeneity one actions on noncompact symmetric spaces with a totally geodesic singular orbit}. Tohoku Math.~J.~\textbf{56}(2), (2004), 163--177.

\bibitem{BT07} J.~Berndt, H.~Tamaru, \emph{Cohomogeneity one actions on noncompact
    symmetric spaces of rank one}, \emph{Trans.\ Amer.\ Math.\ Soc.}\ \textbf{359}(7), (2007). 3425--3438.

\bibitem{BTV95} J.~Berndt, F.~Tricerri, L.~Vanhecke, \emph{Generalized Heisenberg  groups and Damek-Ricci harmonic spaces}, Lecture Notes in Mathematics \textbf{1598}, Springer-Verlag, Berlin (1995).

\bibitem{cls}  M.~Castrill\'{o}n L\'{o}pez, P.M.~Gadea, A.F.~Swann, \emph{Homogeneous structures on real and complex hyperbolic spaces.} Illinois J.\ Math.\ \textbf{53}(2), (2009)  561--574.

\bibitem{ChS} C.~Chevalley, R.D.~Schafer, \emph{The exceptional simple Lie algebras F4 and E6}. Proceedings of the National Academy of Sciences, \textbf{36}(2), (1950) 137--141.

\bibitem{CS02} J.H.~Conway, D.A.~Smith, \emph{On quaternions and octonions: their geometry, arithmetic and symmetry}, Peters, Wellesley Mass.\ (2002).

\bibitem{Ch73} S.-S.~Chen, \emph{On subgroups of the noncompact real exceptional
    Lie group~$\mathrm{F}^*_4$}, \emph{Math.\ Ann.} \textbf{204}, (1973)
    271--284.

\bibitem{Da85} J.\ Dadok, \emph{Polar coordinates induced by actions of compact Lie groups}, Trans.\ Amer.\ Math.\ Soc. \textbf{288}(1), (1985) 125--137.

\bibitem{datri} J.E.~D'Atri, \emph{Certain isoparametric families of hypersurfaces in symmetric spaces}, J.\ Differential Geom. \textbf{14}(1), (1979) 21--40.

\bibitem{D12} J.C.~D\'{\i}az-Ramos, \emph{Polar actions on symmetric spaces.} In: Recent Trends in Lorentzian Geometry. Springer, New York, NY 315--334 (2012).

\bibitem{DDK12} J.C.\ D\'{\i}az-Ramos, M.\ Dom\'{\i}nguez-V\'{a}zquez, A.\ Kollross, \emph{Polar actions on complex hyperbolic space.}  Math.\ Z., \textbf{287}(3), (2011) 1183--1213.

\bibitem{DK10} J.C.\ D\'{\i}az-Ramos, A.\ Kollross, \emph{Polar actions with a fixed point.} \emph{Differential Geom.\ Appl.}\ \textbf{29}(1), (2011) 20--25.

\bibitem{draper} C.\ Draper, \emph{Models of the Lie algebra~$\rm F_4$}, Linear Algebra and its Applications, {\bf 428}(11-12), (2008) 2813--2839.

\bibitem{dynkin1} E.B.\ Dynkin, \emph{Semisimple subalgebras of the semisimple Lie algebras} (Russian). Mat.\ Sbornik \textbf{30}(72), (1952) 349--462 ; English translation: Amer.\ Math.\ Soc.\ Transl.\ Ser.~2, \textbf{6}, (1957) 111--244.

\bibitem{dynkin2} E.B.\ Dynkin, \emph{The maximal subgroups of the classical groups} (Russian). Trudy Mosk.\ Mat.\ Obshch.\ {\bf 1} (1952) 39--166; English translation: Amer.\ Math.\ Soc.\ Transl.\ Ser.~2, {\bf 6} (1957) 245--378.

\bibitem{figueroa} J.\ Figueroa-O'Farrill, \emph{A Geometric construction of the exceptional Lie algebras F 4 and E 8. Communications in Mathematical Physics}, \textbf{283}(3), (2008) 663--674.

\bibitem{freudenthal} H.~Freudenthal, \emph{Oktaven, Ausnahmegruppen und Oktavengeometrie.} (German), Geom.\ Dedicata \textbf{19}(1), doi:~10.1007/BF00233101 (1985) 7--63.

\bibitem{galaev} A.~Galaev, \emph{Isometry groups of Lobachevskian spaces, similarity transformation groups of Euclidean spaces and Lorentzian holonomy groups.} Rend.\ Circ.\ Mat.\ Palermo\ (2) Suppl. No. 79 (2006), 87--97.

\bibitem{gk16} C.~Gorodski,  A.~Kollross. \emph{Some remarks on polar actions.} Annals of Global Analysis and Geometry \textbf{49}(1), (2016) 43--58.

\bibitem{helgason} S.~Helgason, \emph{Differential geometry, Lie
    groups and symmetric spaces}. Academic Press (1978).

\bibitem{hsl}  W.Y.~Hsiang, H.B.~Lawson, \emph{Minimal submanifolds of low cohomogeneity}, J. Differential Geom. \textbf{5}, (1971) 1--38.

\bibitem{iwata} K.~Iwata, \emph{Compact transformation groups on rational cohomology Cayley projective planes}. T\^{o}hoku Math.\ J.(2)~\textbf{33}(4), (1981) 429--422.

\bibitem{killing} W.~Killing, \emph{Die Zusammensetzung der stetigen endlichen Transformationsgruppen} I, Math.\ Ann.\ \textbf{31}, (1888) 252--290. II, \textbf{33}, (1889) 1--48. III, \textbf{34}, (1889) 57--122. IV \textbf{36}, (1890) 161--189.

\bibitem{k02} A.\ Kollross, \emph{A classification of hyperpolar and cohomogeneity
    one actions.} \emph{Trans.\ Amer.\ Math.\ Soc.}\ \textbf{354}(2), (2002)
    571--612.

\bibitem{k07} A.\ Kollross, \emph{Polar actions on symmetric spaces.}
    J.\ Differential Geom.\ \textbf{77}(3),  (2007) 425--482.

\bibitem{k11} A.\ Kollross, \emph{Duality of symmetric spaces and polar
    actions}. J.\ of Lie Theory \textbf{21}(4), (2011) 961--986.

\bibitem{k09} A.\ Kollross, \emph{Low cohomogeneity and polar actions on exceptional compact Lie groups}. Transformation Groups, \textbf{14}(2),  (2009) 387-415.

\bibitem{KL12} A.\ Kollross, A.\ Lytchak, \emph{Polar actions on symmetric spaces of higher rank}, Bulletin of the London Mathematical Society, \textbf{45}(2), (2013) 341--350.

\bibitem{KP} A.\ Kollross, F.\ Podest\`a. \emph{Homogeneous spaces with polar isotropy.} manuscripta mathematica \textbf{110}(4), (2003) 487--503.

\bibitem{lm} H.B.~Lawson, M.-L.~Michelsohn,
    \emph{Spin Geometry}, Princeton University Press, Princeton (1989).

\bibitem{L11} A.\ Lytchak, \emph{Polar foliations of symmetric spaces},
     Geometric and Functional Analysis \textbf{24}(4),  (2014) 1298--1315.

\bibitem{ms} A.~Moroianu, U.~Semmelmann, \emph{Invariant four-forms and symmetric pairs}. Annals of Global Analysis and Geometry, \textbf{43}(2), (2013) 107--121.

\bibitem{murakami} S.~Murakami, \emph{Exceptional simple Lie groups and related topics in recent differential geometry.} Differential Geometry and Topology, (Jiang B., Peng Ch.-K., Hou Z., eds.), Lecture Notes in Math. \textbf{1369}, Springer (1989) 183--221.

\bibitem{oniscik} A.L.\ Oni\v{s}\v{c}ik, \emph{Inclusion relations among transitive compact transformation groups} (Russian). Trudy Mosk.\ Mat.\ Obshch.\ \textbf{11} (1962) 199--242; English translation: Amer.\ Math.\ Soc.\ Transl.\ Ser.~2, \textbf{50} (1966) 5--58.

\bibitem{PaTe} R.\ Palais, C.-L.\ Terng, \emph{A general theory of canonical forms}. Trans.\ Amer.\ Math.\ Soc. \textbf{300}(2), (1987) 771--789.

\bibitem{PT99} F.~Podest\`a, G.~Thorbergsson, \emph{Polar actions on rank-one symmetric spaces}, \emph{J.\ Differential Geom.} \textbf{53}(1), (1999) 131--175.

\bibitem{takagi}  R.~Takagi, \emph{On homogeneous real hypersurfaces in a complex projective space}, Osaka J.~Math.~\textbf{10},  (1973) 495--506.

\bibitem{Th05} G.~Thorbergsson, \emph{Transformation groups and submanifold geometry}. Rend. Mat. Appl.(7), \textbf{25}(1), (2005) 1--16.

\bibitem{Th10} G.~Thorbergsson, \emph{Singular Riemannian foliations and isoparametric submanifolds.} Milan journal of mathematics \textbf{78}(1), (2010) 355--370.

\bibitem{Wo63} J.A.\ Wolf, \emph{Elliptic spaces in Grassmann manifolds},
    \emph{Illinois J.\ Math.}\ \textbf{7}, (1963) 447--462.

\bibitem{wu}  B.~Wu, \emph{Isoparametric submanifolds of hyperbolic spaces.} Trans.\ Amer.\ Math.\ Soc.\ \textbf{331}(2),  (1992) 609--626.

\bibitem{yokota} I.~Yokota, \emph{Exceptional Lie groups.} arXiv preprint arXiv:0902.0431 (2009).

\end{thebibliography}
\end{document}